\documentclass[11pt,reqno,a4paper]{amsart}
\usepackage{geometry}
\usepackage[citecolor=blue, urlcolor=blue, linkcolor=red, colorlinks=true, bookmarksopen=true]{hyperref}
\usepackage{amsmath,amssymb,amsthm,amsfonts,upgreek,bm,mathtools,lmodern,microtype}
\usepackage[usenames,dvipsnames]{color}
\usepackage{url}
\usepackage{enumerate}
\usepackage[active]{srcltx}
\mathtoolsset{centercolon,showonlyrefs,showmanualtags}

\makeatletter
 \def\@textbottom{\vskip \z@ \@plus 9.52pt}
 \let\@texttop\relax
\makeatother

\let\originalleft\left
\let\originalright\right
\renewcommand{\left}{\mathopen{}\mathclose\bgroup\originalleft}
\renewcommand{\right}{\aftergroup\egroup\originalright}
\DeclarePairedDelimiter{\abs}{\lvert}{\rvert}
\DeclarePairedDelimiter{\bra}{(}{)}

\DeclarePairedDelimiter{\pra}{[}{]}
\DeclarePairedDelimiter{\set}{\{}{\}}
\DeclarePairedDelimiter{\ip}{\langle}{\rangle}
\DeclareMathAlphabet{\mathup}{OT1}{\familydefault}{m}{n}
\newcommand{\dx}[1]{\mathop{}\!\mathup{d} #1}
\newcommand{\pderiv}[3][]{\frac{\mathup{d}^{#1} #2}{\mathup{d} #3^{#1}}}

\newcommand{\eps}{\varepsilon}

\newcommand{\R}{\mathbf{R}}

\newcommand{\cA}{\ensuremath{\mathcal A}}
\newcommand{\cB}{\ensuremath{\mathcal B}}

\newcommand{\cF}{\ensuremath{\mathcal F}}

\newcommand{\cI}{\ensuremath{\mathcal I}}

\newcommand{\cK}{\ensuremath{\mathcal K}}

\newcommand{\cO}{\ensuremath{\mathcal O}}
\newcommand{\cP}{\ensuremath{\mathcal P}}

\newcommand{\cW}{\ensuremath{\mathcal W}}
\newcommand{\cX}{\ensuremath{\mathcal X}}

\DeclareMathOperator{\arctanh}{arctanh}

\DeclareMathOperator{\Hess}{Hess}

\DeclareMathOperator{\tI}{I}
\DeclareMathOperator{\tII}{II}
\DeclareMathOperator{\tIII}{III}
\DeclareMathOperator{\tIV}{IV}

\DeclareMathOperator{\Ric}{Ric}
\DeclareMathOperator{\Lip}{Lip}

\DeclareMathOperator{\MLSI}{MLSI}

\DeclareMathOperator{\Glauber}{Glauber}

\DeclareMathOperator{\CE}{CE}

\newtheorem{theorem}{Theorem}[section]

\newtheorem{lemma}[theorem]{Lemma}
\newtheorem{proposition}[theorem]{Proposition}
\newtheorem{assumption}[theorem]{Assumption}
\theoremstyle{definition}
\newtheorem{definition}[theorem]{Definition}
\newtheorem{Example}[theorem]{Example}
\theoremstyle{remark}
\newtheorem{remark}[theorem]{Remark}

\numberwithin{equation}{section}
\begin{document}
\title[Curvature for discrete mean-field dynamics]{Entropic curvature and convergence to equilibrium for mean-field dynamics on discrete spaces}

\author{Matthias Erbar}
\address{Insitut f\"ur Angewande Mathematik, Endenicher Allee 60, Universit\"at Bonn }
\email{erbar@iam.uni-bonn.de}
\author{Max Fathi}
\address{CNRS \& Institut de Math\'ematiques de Toulouse \\ Universit\'e de Toulouse \\
118 route de Narbonne, 31062 Toulouse, France}
\email{max.fathi@math.univ-toulouse.fr}
\author{Andr\'e Schlichting}
\address{Insitut f\"ur Angewande Mathematik, Endenicher Allee 60, Universit\"at Bonn }
\email{schlichting@iam.uni-bonn.de}
%
%
%
%\subjclass[2000]{ }
%
%
%
%\keywords{}
%
%
%
\date{\today}
\begin{abstract}
  We consider non-linear evolution equations arising from mean-field
  limits of particle systems on discrete spaces. We investigate a
  notion of curvature bounds for these dynamics based on convexity of
  the free energy along interpolations in a discrete transportation
  distance related to the gradient flow structure of the
  dynamics. This notion extends the one for linear Markov chain
  dynamics studied in \cite{EM11}. We show that positive curvature
  bounds entail several functional inequalities controlling the
  convergence to equilibrium of the dynamics. We establish explicit
  curvature bounds for several examples of mean-field limits of various
  classical models from statistical mechanics.
\end{abstract}
\maketitle
%
%
%
% \tableofcontents
%
%

\section{Introduction}
This work is about long-time behavior for mean-field systems on
discrete spaces. Mean-field equations describe the large-scale limit
of interacting particle systems where the total force exerted on any
given particle is the average of the forces exerted by all other
particles on the tagged particle. They are used to describe collective
behavior in many areas of sciences. Examples include the modeling of
granular flows in physics \cite{Benedetto1997} and collective behavior
and self-organization for groups of animals \cite{Degond2014,
  Cattiaux2018}. We refer to \cite{Sznitman1989} for an introduction
to the mathematical theory.

One of the important questions in the mathematical analysis of these equations is their long-time
behavior. In \cite{Carrillo2003}, Carrillo, McCann and Villani
obtained quantitative bounds on the rate of convergence to equilibrium
for McKean-Vlasov equations in a continuous setting of the form
\begin{equation*}
\partial_t\rho = \nabla \cdot[\rho \nabla(S'(\rho) + V + W*\rho)]
\end{equation*}
under strong convexity assumptions on the potentials $S$, $V$ and
$W$. The core idea underlying their method was the fact that the PDE has a
gradient flow structure, i.e.~it can be recast as a gradient descent
equation $\partial_t\rho=-\nabla F(\rho)$ of the free energy
functional $F(\rho)=\int S(\rho)+\int V\dx\rho+\int W*\rho \dx\rho$ in the
space of probability measure with respect to the Kantorovitch-Wasserstein
distance $W_2$, which has a formal Riemannian description via Otto
calculus \cite{Otto2001, OV00}. The use of such structures in the
study of long-time behavior comes from the fact that as soon as the
driving functional satisfies some uniform convexity property (with
respect to the particular metric structure), it must decay
exponentially fast towards its minimal value along solutions of the
evolution equation. Moreover, we can use convexity to derive strong
functional inequalities relating the distance, the entropy functional
and the entropy dissipation functional \cite{OV00}.

\subsection{Setup and main results}
Our main motivation here is to adapt the approach of
\cite{Carrillo2003} to mean-field equations in a discrete setting.  We
consider discrete mean-field dynamics of the form
\begin{equation}\label{e:master}
 \dot{\mu}(t) = \mu(t)Q\bra*{\mu(t)} \;,  
\end{equation}
where $\mu$ is a flow of probability measures on a finite set
$\mathcal{X}$ and $(Q(\mu)_{xy})_{x,y\in\cX}$ is a parametrized
collection of Markov kernels. These dynamics naturally arise as
scaling limits of interacting particles systems on graphs where the
interaction only depends on the normalized empirical measure of the
system (which indeed corresponds to mean-field interactions). They
generalize linear Markov chains on discrete spaces, which orrepsond to
the case where $Q$ is a constant Markov kernel, independent of $\mu$.

While the Wasserstein gradient flow approach works well on continuous
spaces, it fails in the discrete setting, since the Wasserstein
$L^2$-transport distance does not admit any non-trivial absolutely
continuous curve. In our previous work \cite{EFLS16}, we derived a
gradient flow structure for \eqref{e:master} by replacing the role of
the Wasserstein distance with a distance $\cW$ constructed via a
suitable modification of the Benamou-Brenier formula for optimal
transport, extending similar earlier results for linear reversible
Markov chains obtained in \cite{Maas2011, Mielke2013, Chow2012}.
Under the condition that the rates $Q$ are Gibbs with a potential
$K: \cP(\cX)\times \cX \to \R$ (see
Assumption~\ref{ass:GibbsPotential}), i.e. $Q(\mu)$ is reversible with respect to a local Gibbs measure of the form $\pi_x(\mu)=Z(\mu)^{-1}\exp\bigl(-H_x(\mu)\bigr)$, with $H_x$ given in terms of the potential $K$, we showed that this dynamic is
the gradient flow of the free energy functional 
\begin{equation}\label{e:def:FreeEnergy}
 \cF(\mu)=\sum_{x\in\mathcal{X}}\mu_x \log\mu_x + U(\mu), \qquad\text{with}\qquad U(\mu)=\sum_{z\in\mathcal{X}}\mu_z\, K_z(\mu),
\end{equation}
with respect to the distance $\cW$ on the simplex of probability
measures on $\mathcal{X}$, see Proposition~\ref{prop:GF}. This built
up on previous works \cite{BDFR15a, BDFR15b} that showed that
$\mathcal{F}$ is indeed a Lyapunov functional for the flow.  An
archetypical example, which we shall discuss in some details later on,
is the classical Curie-Weiss model, which corresponds to a mean-field
dynamic on a two-point space. Already this easy model exhibits
interesting behavior, such as a phase transition at an explicit
critical value of a temperature parameter.
\smallskip

In the present work, we exploit this gradient flow structure to
analize the long-term behaviour of \eqref{e:master} inspired by the
approach in \cite{Carrillo2003} by investigating convexity properties
of the free energy along discrete optimal transport paths for a
non-linear Markov triple $(\cX,Q,\pi)$ as above. Following the works
of Lott, Sturm and Villani \cite{LV09, S06} for metric measure spaces
and \cite{EM11, Mielke2013} for linear Markov chains, we make the following

\begin{definition}[Entropic Ricci curvature lower bound]
  We say that $(\cX,Q,\pi)$ has \emph{Ricci curvature bounded below by
    $\kappa\in\R$} (for short $\Ric(\cX,Q)\geq\kappa$) if for any
  $\cW$-geodesic $(\mu_t)_{t\in[0,1]}$:
  \begin{align*}
    \cF(\mu_t)\leq (1-t)\cF(\mu_0) + t\cF(\mu_1) -\frac{\kappa}{2}t(1-t)\cW(\mu_0,\mu_1)^2\;.
  \end{align*}
\end{definition}
We will show, see Theorem \ref{thm:Ric-equiv}, that Ricci curvature lower bounds can be characterized in
terms of a discrete Bochner-type inequality by deriving the Hessian of
$\cF$ in the Riemannian structure $\cW$, as well as in terms of the Evolution Variational inequality EVI$_\kappa$ for the solutions to \eqref{e:master}: 
\begin{align*}
  \frac{1}{2}\frac{\mathup{d}^+}{\mathup{d}t} \cW(\mu_t, \nu)^2 + \frac{\kappa}{2}
\cW( \mu_t, \nu)^2 \leq \cF(\nu) - \cF(\mu_t)\;.
\end{align*}
Further, we will show that a positive lower bound on the Ricci
curvature entails a number of functional inequalities that control the
convergence to equilibrium of the mean-field systems. These involve a discrete Fisher information functional $\cI:\cP(\cX)\rightarrow[0,\infty]$ given by
  \begin{align*}%\label{e:intro:dissipation}
    \mathcal{I}(\mu) = \frac{1}{2} \sum\limits_{x,y} \Theta\bra*{\mu_{x}Q_{xy}(\mu), \mu_{y}Q_{yx}(\mu)}\;,\qquad \Theta(a,b)=(a-b)(\log a-\log b)\;,
  \end{align*}
which arises from the dissipation of $\cF$ along solutions to \eqref{e:master} as $\frac{d}{dt} \cF(\mu_t)=-\cI(\mu_t)$.
One of our main results is the following theorem which can be seen as a discrete analog of \cite[Thm.~2.1]{Carrillo2003}.
\begin{theorem}\label{thm:intro-CMV}
 Assume that $\Ric(\cX,Q,\pi)\geq \lambda$ for some $\lambda>0$. Then the following hold:
 \begin{itemize}
 \item[(i)] there exists a unique stationary point $\pi^*$ for the evolution \eqref{e:master}, it is the unique minimizer of $\cF$. Let $\cF_*(\cdot):=\cF(\cdot)-\cF(\pi^*)$;
 \item[(ii)] the \emph{modified logarithmic Sobolev inequality} with constant $\lambda>0$ holds, i.e.~for all $\mu\in \cP(\cX)$,
    \begin{align*}\tag*{$\MLSI(\lambda)$}
      \cF_*(\mu)\leq \frac{1}{2\lambda}\cI(\mu)\;;
    \end{align*}
\item[(iii)] for any solution $(\mu_t)_{t\geq0}$ to \eqref{e:master} we have exponential decay of the free energy:
  \begin{align*}
    \cF_*(\mu_t)\leq e^{-2\lambda t} \cF_*(\mu_0)\;;
  \end{align*}
 \item[(iv)] the \emph{entropy-transport inequality} with constant $\lambda>0$ holds, i.e.~for all $\mu\in\cP(\cX)$,
    \begin{align*}\tag*{ET$(\lambda)$}
      \cW(\mu,\pi_*) \leq \sqrt{\frac{2}{\lambda}\cF_*(\mu)}\;.
    \end{align*}
  % \item[(3)] a \emph{Poincar\'e} inequality if for all $\mu\in\cP(\cX)$ and $\phi\in\R^\cX$ such that $\sum_x\phi(x)\pi(\mu;x)=0$, 
  %   \begin{align}\tag{P$(\lambda)$}\label{eq:P}
  %     \sum_x|\phi(x)|^2\pi(\mu;x) \leq \frac{1}{\lambda} \frac12\sum_{x,y}|\nabla\phi(x,y)|^2Q(\mu;x,y)\pi(\mu;x)\;.
  %   \end{align}
 \end{itemize}
\end{theorem}

\subsection{Examples}
We will establish explicit curvature bounds for several examples of
(relatively simple) mean-field dynamics, such as the Curie-Weiss
model, zero-range mean-field dynamcis and misanthrope processes. We
compute a formula for the second derivative of entropy along
geodesics, and generalize techniques developed in \cite{FM16, EHMT17}
to the present non-linear situation in order to bound for bounding
this second derivative. The nonlinearity of the dynamic gives rise to
several extra terms when computing the Hessian of the free energy
functional, which complicates the analysis.  \smallskip 

In the case of
the Curie--Weiss model, we will show that a positive lower curvature
bound holds down to the critical temperature, see Section
\ref{sec:CW}.

Another particular family of dynamic we shall be interested is when
the flux of particles from some site $x$ to a site $y$ is a function
of the particle density at site $y$, that is $f(c_y)$. In the
situation where $f$ is constant, this would correspond to the scaling
limit of independent particles on the complete graph. As in
\cite{FM16, EHMT17}, our approach is in some sense perturbative in
nature, and we shall consider rates of the form $f(r) = T + g(r)$, and
show that if $g$ is not too large in some sense, relative to $T$, then
we can derive a rate of convergence to equilibrium. This is inspired
by recent work of Villemonais \cite{Villemonais2017}, who proved that
the $N$ particle system has a positive Ollivier-Ricci (or coarse
Ricci) curvature (another notion of curvature, corresponding to a
contraction rate for the Markovian dynamic) independently of the
system size, and hence converges to equilibrium in $L^2$ distance, via
a uniform estimate on the Poincar\'e constant of the dynamic. Our
approach has the advantage of yielding rates of convergence in
relative entropy via Theorem \ref{thm:intro-CMV}, which is a strictly
stronger notion of convergence. 

\subsection{Connection to the literature}
The approach of \cite{Carrillo2003} was later extended to other
potentials \cite{Carrillo2006,Bolley2013}. Other approaches developed
later include using uniform convergence estimates for a stochastic
particle approximation \cite{Cattiaux2008} and coupling arguments
\cite{Eberle2016, Eberle2016+}. Without convexity, deriving rates of
convergence can be quite delicate, since there may be multiple
equilibria \cite{Tugaut2013}, unlike what happens for linear
diffusions.

Our approach developed here builds on earlier work \cite{Maas2011,
  Mielke2013, Chow2012} contructing gradient flow structures for
linear Markov chain dynamics and \cite{EM11} studying Ricci curvature
and its impact on functional inequalities in this context. It is also
related to, but different from the one developped in
\cite{Kraaij2016}, which uses convexity of the entropy along a
different type of paths, the so-called entropic interpolations, rather
than geodesic paths, to establish functional inequalities involving
relative entropy. In the continuous setting, entropic interpolations
are regularizations of geodesics in Wasserstein space, but in the
discrete case it seems that the entorpic interpolations of
\cite{Kraaij2016} are related to a gradient flow structure different from the one of \cite{EFLS16} we use here.

\subsection*{Organization}
The plan of the paper is as follows. Section~\ref{s:setup} introduces the
mathematical framework we shall work in. Section~\ref{s:curv} introduces the
notion of curvature bounds in our setting, and contains the
computation of the Hessian for general dynamics. Section~\ref{sec:consequ}
investigates the consequences of Ricci curvature bounds in terms of
functional inequalities and convergence to equilibrium for the
nonlinear dynamics. Finally, Section~\ref{s:bounds} investigates curvature bounds
for several examples of mean-field dynamics inspired by classical
models of statistical physics.

\section{Setup}\label{s:setup}

\subsection{Gradient-flow formulation}

The main definitions and results from~\cite{EFLS16} on which this work
will build are collected in this section. The gradient flow structure
of~\eqref{e:master} is based on the existence of a suitable potential,
which is ensured by the following constraint, which we shall assume
to hold throughout the article. We recall that a rate matrix $Q$ of a Markov chain in the continuous time setting 
satisfies 
\begin{align*}
  \forall x\ne y: Q_{xy}\geq 0 \qquad\text{and}\qquad Q_{xx} = -\sum_{y\ne x} Q_{xy} \;.
\end{align*}
\begin{assumption}\label{ass:GibbsPotential}
  Let $K:\mathcal{P}(\mathcal{X})\times\mathcal{X}\rightarrow\R$ be
  such that for each
  $x\in \mathcal{X}, K_{x}:\mathcal{P}(\mathcal{X})\rightarrow\R$ is a
  twice continuously differentiable. Let $\{Q(\mu)\in \R^{\cX\times \cX}\}_{\mu\in\cP(\cX)}$
  be a family of
  rate matrices that \emph{is Gibbs with respect to the  potential function $K$}, i.e.~for each
  $\mu\in \cP(\cX)$, $Q(\mu)$ is the rate matrix of an irreducible,
  reversible ergodic Markov chain with respect to the probability
  measure
  \begin{equation}\label{e:def:piH}
    \pi_{x}(\mu) = \frac{1}{Z(\mu)}\exp\bra[\big]{-H_{x}(\mu)}\;,
  \end{equation}
  with
  \begin{equation}\label{e:def:potentialU}
      H_{x}(\mu) = \frac{\partial}{\partial \mu_x}U(\mu)\qquad\text{and}\qquad U(\mu)=\underset{x \in \cX}{\sum} \; \mu_x K_x(\mu)\;.
  \end{equation}
  In particular $Q(\mu)$ satisfies the detailed balance condition
  with respect to $\pi(\mu)$, that is for all~$x,y\in \cX$
  \begin{equation}\label{e:DBC}
    \pi_{x}(\mu) Q_{xy}(\mu) = \pi_{y}(\mu) Q_{yx}(\mu)
  \end{equation}
  holds. Moreover, we assume that for each $x,y\in \cX$ the map
  $\mu\mapsto Q_{xy}(\mu)$ is Lipschitz continuous over~$\cP(\cX)$.
\end{assumption}
We will refer to the triple $(\cX,Q,\pi)$ as above for short as a
\emph{non-linear Markov triple}.

The specific form of~\eqref{e:def:piH} with~\eqref{e:def:potentialU}
emerges from the detailed balance condition of an underlying
$N$-particle system, from which the dynamics we are interested arise
in the limit $N\to \infty$ (see \cite{EFLS16}). Associated to a
non-linear Markov triple $(\cX,Q,\pi)$ is the non-linear \emph{master
  equation}
\begin{equation}\label{e:master2}
 \dot{\mu}(t) = \mu(t)Q\bra*{\mu(t)} \;,
\end{equation}
which is the deterministic evolution equation describing the
mean-field limit of the underlying particle system. Based on the above
assumption a gradient flow formulation of~\eqref{e:master2} is
established in~\cite[Proposition 2.13]{EFLS16} as we shall briefly
recall.

Consider the \emph{Onsager operator} $\cK[\mu]: T_\mu^*\cP(\cX) \to T_\mu \cP(\cX)$ given by
\begin{equation*}
 \cK[\mu]\psi(x) := -\frac{1}{2}\sum_{y} \Lambda\bra[\big]{ \mu_x Q_{xy}(\mu), \mu_y Q_{yx}(\mu)}\bigl(\psi(y)-\psi(x)\bigr)\;,
\end{equation*}
where
$\Lambda(a,b) = \int_0^1 a^{1-s} b^s \dx{s} = \frac{a-b}{\log a - \log
  b}$
is the logarithmic mean. Then the master equation can be written in
gradient flow form using the functional $\cF$ from
\eqref{e:def:FreeEnergy}:
\begin{equation}\label{e:GF}
 \frac{d\mu_{t}}{dt} = - \cK[\mu_t] D\cF(\mu_t)\;.
\end{equation}
In other words, \eqref{e:master2} is the gradient flow of $\cF$
with respect to the Riemannian structure on $\cP(\cX)$ induced by the metric
tensor $\cK[\mu]^{-1}$. Since this Riemannian metric degenerates at
the boundary of $\cP(\cX)$ we note the following characterization in
metric terms. We consider the distance function on
$\mathcal{P}(\mathcal{X})$ that is formally induced by the Riemannian
metric $\cK[\mu]^{-1}$, i.e. for $\mu_0,\mu_1\in\cP(\cX)$ we set
\begin{equation*}
\mathcal{W}(\mu_0, \mu_1) := \underset{(\mu, \psi) \in CE}{\inf}  \left(\int_0^1{\mathcal{A}(\mu_t, \psi_t)dt}\right)^{1/2}
\end{equation*}
where CE is the set of curves $(\mu_t, \psi_t)_{t\in[0,1]}$ with $t \rightarrow \mu_t$ continuous, $\psi$ measurable and integrable in time, and satisfying the continuity equation
\begin{equation}\label{e:CE}
\dot{\mu}_t = \cK[\mu_t]\psi_t\;
\end{equation}
in distribution sense, and the action functional $\cA$ is given by
\begin{equation}\label{e:def:MF:action}
 \mathcal{A}(\mu,\psi)=\langle \psi, \cK[\mu]\psi\rangle
=\frac{1}{2}\sum_{x,y}(\psi(y)-\psi(x))^{2}\, \Lambda\bra*{ \mu_x Q_{xy}(\mu), \mu_y Q_{yx}(\mu)},
\end{equation}
\begin{proposition}[Gradient flow structure of the mean-field system]\label{prop:GF}
  Let $(\cX,Q,\pi)$ be a non-linear Markov triple satisfying
  Assumption~\ref{ass:GibbsPotential}. Then any solution to
  \eqref{e:master2} is a gradient flow of $\cF$ with respect to the distance $\cW$.
\end{proposition}
The distance $\cW$ and the above gradient flow structure are
extensions of the discrete transport distance constructed in
\cite{Maas2011} and the gradient flow structure of linear Markov chains
to the non-linear case. See \cite[Section 2.3]{EFLS16} for more
background on the construction of the distance $\cW$.

An immediate consequence of the gradient flow formulation~\eqref{e:GF}
is the free energy dissipation relation established in~\cite[Remark
2.14]{EFLS16}:
\begin{equation}\label{e:FED}
  \cF(\mu_t) + \int_{0}^t \cI(\mu_s) \dx{s} = \cF(\mu_0) \qquad\text{for any } t>0 \;.
\end{equation}
Here, the discrete Fisher information or dissipation
$\cI:\cP(\cX)\rightarrow[0,\infty]$ is defined  by
  \begin{equation}\label{e:def:dissipation}
    \mathcal{I}(\mu) = \begin{cases} 
                         \frac{1}{2} \sum\limits_{(x,y)\in E_{\mu}} \Theta\bra*{\mu_{x}Q_{xy}(\mu), \mu_{y}Q_{yx}(\mu)} \;, & \text{ for } \mu\in\cP^*(\cX) \\
                         +\infty  \;, & \text { else}
                       \end{cases} \;,
  \end{equation}
  with
%\begin{equation}\label{e:def:Theta}
    $\Theta: \R_+ \times \R_+ \to \R_+$ defined by  $\Theta(a,b) = (a-b)(\log a - \log b)$. 
%\end{equation}
In this framework, the Fisher information can be reinterpreted as the squared modulus of the gradient of the entropy with respect to the discrete transport metric $\mathcal{W}$, i.e.~we have 
\begin{align*}
  \cI(\mu)=\langle D\cF(\mu),\cK[\mu]D\cF(\mu)\rangle\;.
\end{align*}

\subsection{Notation}
We will use the following notation throughout the paper.

Given a function $\psi\in\R^\cX$ we will denote by $\nabla\psi\in\R^{\cX\times\cX}$ its discrete gradient, given by
\begin{align*}
  \nabla\psi_{xy}=\psi_y-\psi_x\;.
\end{align*}
For a function $\Psi\in\R^{\cX\times\cX}$ we denote by $\nabla\cdot\Psi$ its discrete divergence, given by
\begin{align*}
  (\nabla\cdot\Psi)_x=\frac12\sum_{y\in\cX}\Psi_{xy}-\Psi_{yx}\;.
\end{align*}
For $\psi,\phi\in \R^\cX$ and $\Psi,\Phi\in\R^{\cX\times\cX}$ we will denote the Euclidean inner products by
\begin{align*}
  \langle\psi,\phi\rangle=\sum_{x\in\cX}\psi_x\phi_x\;,\qquad  \langle\Psi,\Phi\rangle=\frac12\sum_{x,y\in\cX}\Psi_{xy}\Phi_{xy}\;.
\end{align*}
Then we have the integration by parts formula
\begin{align*}
  \langle\psi,\nabla\cdot\Phi\rangle=-\langle\nabla\psi,\Phi\rangle\;.
\end{align*}
For a functions $\Phi,\Psi$ in $\R^{\cX\times\cX}$, we denote by
$\Phi\cdot\Psi$ the componentwise product. Using the shorthand
notation
$\Lambda(\mu)_{xy}:=\Lambda\bigl(\mu_xQ(\mu)_{xy},\mu_yQ(\mu)_{yx}\bigr)$
we can thus write the continuity equation \eqref{e:CE} and the action functional \eqref{e:def:MF:action} compactly as
\begin{align*}
  \dot\mu_t+\nabla\cdot\bigl(\Lambda(\mu_t)\cdot\nabla\psi_t\bigr)=0\;,\qquad
  \cA(\mu,\psi)=\langle\nabla\psi,\Lambda(\mu)\cdot\nabla\psi\rangle\;.
\end{align*}
We will switch freely between notations for the components of
functions $\psi\in\R^\cX$, $\Psi\in\R^{\cX\times\cX}$ as
$\psi_x,\Psi_{xy}$ or $\psi(x),\Psi(x,y)$ depending on what is more
readable in the presence of other indices, e.g.~a time parameter $t$.

\subsection{Equilibria and qualitative longtime behavior}
From the gradient flow formulation, it is straightforward to obtain
the following characterization of stationary states, which is
completely analog to the McKean-Vlasov equation on $\R^n$
\cite[Proposition 2.4 and Corollary 2.5]{CGPS18}.
\begin{proposition}[Characterization of stationary points]\label{prop:charact:stationary}
  Let $(\cX,Q,\pi)$ be a non-linear Markov triple satisfying
  Assumption~\ref{ass:GibbsPotential}. Then, the following statements
  are equivalent: \begin{enumerate}
   \item $\pi^*$ is a stationary solution to~\eqref{e:master}, that is $\pi^* Q(\pi^*) = 0$.
   \item $\pi^*$ is a fixed point of the map $\mu\mapsto \pi(\mu)$~\eqref{e:def:piH}, that is $\pi^* = \pi(\pi^*)$. 
   \item $\pi^*$ is a critical point of $\cF$~\eqref{e:def:FreeEnergy} on $\cP(\cX)$.
   \item $\pi^*$ is a global minimizer of $\cI$~\eqref{e:def:dissipation}, that is $\cI(\mu^*)=0$.
  \end{enumerate}
  The set of all stationary points $\pi^*$ is denoted by $\varPi^*$.
  
  Moreover, it holds that $\varPi^* \subset \cP^*(\cX)$, i.e. each stationary point has strictly positive density. 
\end{proposition}
\begin{proof}
\emph{(1)$\Leftrightarrow$(2):} Let $\pi^*Q(\pi^*)=0$. The rate matrix $Q^* = Q(\pi^*)$ is by assumption the rate matrix of an irreducible reversible Markov chain with unique reversible measure $\pi(\pi^*)$. In particular, it is also the unique stationary solution to $\pi(\pi^*) Q^*=0$ and hence $\pi^*= \pi(\pi^*)$. If $\pi^*=\pi(\pi^*)$, we calculate using the local detailed balance condition~\eqref{e:DBC} and find
\[
  \sum_{x\in\cX} \pi^*_{x} Q_{xy}(\pi^*) = \sum_{x\in \cX} \pi_x(\pi^*) Q_{xy}(\pi^*) = \sum_{x\in\cX} \pi_y(\pi^*) Q_{yx}(\pi^*) = 0  \;,
\]
since $Q$ is a rate matrix.

\smallskip

\emph{(2)$\Leftrightarrow$(3):} Take $\mu\in \cP^*(\cX)$ and any $\nu\in \cP(\cX)$. Let $\mu_s = (1-s)\mu + s \nu$ the standard linear interpolation. Then, it holds 
\[
 \left.\pderiv{}{s} \cF(\mu_s) \right|_{s=0} = \sum_{x\in \cX} \bra[\Big]{\log \mu_x - 1 + \partial_{\mu_x} U(\mu)} \bra*{\nu_x -\mu_x} = \sum_{x\in \cX} \log \frac{\mu_x}{\pi_x(\mu)}\bra*{\nu_x -\mu_x}  \;,
\]
where we used the relations~\eqref{e:def:piH} and~\eqref{e:def:potentialU}. Now, if $\mu=\pi^* = \pi(\pi^*)$ the right hand side is zero and hence $\pi^*$ a critical point if $\cF$. On the other hand, if the right hand side is zero for all $\nu \in \cP(\cX)$, it follows that $\mu_x = C \pi_x(\mu)$ for a constant $C$. Since $\mu,\pi(\mu)\in \cP^*(\cX)$, we have that $C=1$ and hence critical points are fixed points. 

\smallskip

\emph{(2)$\Leftrightarrow$(4):} Let $\pi^*=\pi^*(\pi)$. Since $\cI(\mu)\geq 0$ for all $\mu\in \cP(\cX)$, we immediately find from the local detailed balance condition~\eqref{e:DBC} that $\cI(\pi^*)=0$. Likewise, any global minimizer~$\pi^*$ satisfies by the definition of $\cI$ that $\pi^*_x Q_{xy}(\pi^*) = \pi^*_y Q_{yx}(\pi^*)$, that is the local detailed balance condition~\eqref{e:DBC}. Since again by assumption $Q(\pi^*)$ has the unique reversible measure $\pi(\pi^*)$, we conclude that $\pi^*=\pi^*(\pi)$. 

\smallskip

Finally, the positivity follows from the definition of $\pi(\mu)$ in~\eqref{e:def:piH} and the assumptions on~$K$ implying that $H$ is finite. Hence, $\pi(\mu)\in\cP^*(\cX)$ for all $\mu\in \cP(\cX)$ implies in particular that $\pi(\pi^*)=\pi^*\in \cP^*(\cX)$.
\end{proof}
Another useful information provided by the gradient flow information
is the free energy dissipation relation~\eqref{e:FED}, which
immediately shows that $\cF$ is a Lyapunov function for the
evolution~\eqref{e:master}. By standard theory, we can conclude the
following qualitative longtime behavior.
\begin{proposition}[Convergence to stationary points]\label{prop:qualitative:longtime}
  Let $Q$ satisfy Assumption~\ref{ass:GibbsPotential}, then
  $c(t) \to \pi^*$ for some $\pi^*\in \varPi^*$ as $t \to \infty$.
\end{proposition}
\begin{proof}
  The proof follows along standard arguments from the theory of dynamical systems (see for instance~\cite[Section 6]{Teschl}). 
  
  By Assumption~\ref{ass:GibbsPotential}, $Q$ is Lipschitz on
  $\cP(\cX)$, which implies by standard well-posedness for ODEs, that
  the solutions $(\mu_t)_{t\geq 0}$ to~\eqref{e:master} are globally
  defined and generate a semigroup on $\cP(\cX)$. The $\omega$-limit is
  given by
  \[
    \omega(\mu) = \set*{ \nu \in \cP(\cX) : \mu_{t_j} \to \nu \text{ for some sequence } t_j \to \infty } \;.
  \]
  Since $\cP(\cX)$ is compact, each orbit
  $\cO^+(\mu_0) = \bigcup_{t\geq 0} \mu_t$ for any
  $\mu_0 \in \cP(\cX)$ is also compact in $\cP(\cX)$ and the
  $\omega$-limit is non-empty and quasi-invariant, that is for
  $\nu \in \omega(\mu_0)$ it holds
  $\cO^+(\nu) \subseteq \omega(\mu_0)$. 
  Moreover, again thanks to the compactness of $\cP(\cX)$ follows for any
  $\mu_0\in \cP(\cX)$ that
  $\operatorname{dist}_{\cP(\cX)}(\mu_t,\omega(\mu_0)) \to 0$ as
  $t\to \infty$ (see also \cite[Lemma 6.7]{Teschl}). 
  
  Since the free energy functional $\cF$ is continuous on $\cP(\cX)$
  and monotone along the flow, it follows that $\omega(\mu_0)$
  consists of complete orbits along which $\cF$ has the constant value
  $\cF^\infty = \lim_{t\to \infty} \cF(\nu_t)$ with
  $\nu_0\in \omega(\mu_0)$. By the free energy dissipation
  relation~\eqref{e:FED}, it follows that for any
  $\nu_0\in \omega(\mu_0)$ and any $t>0$ we have
  \[
    \cF^\infty + \int_0^t \cI(\nu_s) \dx{s} = \cF^\infty
  \]
  and hence the nonnegativity of $\cI$ and continuity of trajectories
  imply $\cI(\nu_s) = 0$ for all $s\in [0,t]$. Hence, $\omega(\mu_0)$
  consists of all states $\nu$ such that $\cI(\nu)=0$, which by
  Proposition~\ref{prop:charact:stationary} entails
  $\nu \in \varPi^*$ and moreover also that $\nu$ is a stationary solution $\nu Q(\nu)=0$. 
\end{proof}

Our purpose in this work can be summarized as giving sufficient
conditions for which the above statement on convergence to equilibrium
can be made quantitative (but which shall automatically enforce that
$\varPi^*$ contains a single element).

\section{Curvature for non-linear Markov chains}\label{s:curv}

In this section, we introduce a notion Ricci curvature lower bounds
for non-linear Markov chains based on geodesic convexity of the
entropy. This generalizes the notion of curvature for linear Markov
chains developed in \cite{EM11} inspired by the approach of Lott,
Sturm and Villani \cite{LV09, S06} to a synthetic notion of lower
bounds on Ricci curvature for geodesic metric measure spaces.

Let $(\cX,Q,\pi)$ be a non-linear Markov chain according to Assumption~\ref{ass:GibbsPotential}
and let $\cF$ be the associated free energy functional~\eqref{e:def:FreeEnergy}
and $\cW$ the associated transport distance.

\begin{definition}[Entropic Ricci curvature lower bound]
  We say that $(\cX,Q,\pi)$ has \emph{Ricci curvature bounded below by
    $\kappa\in\R$} (for short $\Ric(\cX,Q)\geq\kappa$) if for any
  $\cW$-geodesic $(\mu_t)_{t\in[0,1]}$:
  \begin{align*}
    \cF(\mu_t)\leq (1-t)\cF(\mu_0) + t\cF(\mu_1) -\frac{\kappa}{2}t(1-t)\cW(\mu_0,\mu_1)^2\;.
  \end{align*}
\end{definition}

We will show that a lower bound on the Ricci curvature can be
characterized equivalently by a lower bound on the Hessian of the free
energy functional $\cF$ with respect to the Riemanian structure on $\cP_*(\cX)$
induced by $\cW$, or via an Evolution Variational Inequality for the
non-linear Markov dynamics.

To this end, we first derive the geodesic equation for the distance
$\cW$ as well as an expression for the first variation of the free energy.
 
\begin{lemma}[Geodesic equation]
  Let $(\mu_t)_{t\in [0,1]}$ be a constant speed geodesic contained in
  $\cP_*(\cX)$. Then the unique potential $(\psi_t)_{t\in[0,1]}$ such
  that $(\mu,\psi)\in\CE$ solves
  \begin{align*}
    \dot \psi_t(z) +  \frac{1}{2} \partial_{\mu(z)}\langle\nabla\psi_t,\Lambda(\mu_t)\cdot\nabla\psi_t\rangle=0\;,
  \end{align*}
or explicitely
\begin{equation}\label{e:def:MF:geodesic:vectorfield}
 \dot \psi_t(z) +  \frac{1}{4} \partial_{\mu(z)}\sum_{x,y} (\psi_t(x)-\psi_t(y))^2 \Lambda\bigl(\mu_t(x)Q(\mu_t;x,y),\mu_t(y)Q(\mu_t;y,x)\bigr) = 0\;,
\end{equation}
where $\partial_{\mu(z)}$ is the derivative with respect to $\mu(z)$.
\end{lemma}

 \begin{remark}
   In the case of a linear Markov chain, where $Q$ is independent of
   $\mu$, the expression \eqref{e:def:MF:geodesic:vectorfield}
   simplifies to
 \begin{align*}
   \dot \psi_t(z) +  \frac12\sum_{y} (\psi_t(z)-\psi_t(y))^2 \partial_1\Lambda\bigl(\mu_t(z)Q(z,y),\mu_t(y)Q(y,z)\bigr)Q(z,y) = 0\;,
 \end{align*}
  recovering the geodesic equation derived in \cite[Prop.~3.4]{EM11}.
\end{remark}

\begin{proof}
  Since $\cP_*(\cX)$ is a smooth Riemannian manifold, uniqueness and smoothness of geodesics imply that the curve
  $\mu_t$ is smooth, and that there exists a unique (up to constants) potential
  $\psi_t$ such that $(\mu,\psi)\in\CE$ and achieves in the infimum for the action 
  \begin{align*}
    A(\mu,\psi) = \int_0^1\cA(\mu_t,\psi_t) \dx{t}\;,\qquad \cA(\mu,\psi)=\ip{\nabla\psi,\Lambda(\mu)\cdot\nabla\psi}\;, 
  \end{align*}
	and moreover $\psi$ is then also a smooth curve. 
  We will derive \eqref{e:def:MF:geodesic:vectorfield} as the
  corresponding Euler--Langrange equation. So let
  $\mu^s_t\in\cP_*(\cX)$ for $s\in[-\eps,\eps]$ be a smooth
  perturbation of $\mu$ such that $\mu^s_0=\mu_0$ and $\mu^s_1=\mu_1$
  for all $s$. Let $\psi^s_t$ be the unique potentials such that
  $(\mu^s_\cdot,\psi^s_\cdot)\in\CE$. Note that $\psi^s_t$ is smooth
  in $s$ and $t$. Then we have
  \begin{align}\label{eq:perturb1}
    \left.\pderiv{}{s}\right\vert_{s=0} A(\mu^s,\psi^s) = 0\;.
  \end{align}
  We compute
  \begin{align*}
    \pderiv{}{s} A(\mu^s,\psi^s)
  = 
\int_0^1 2\ip{\nabla\psi_t^s,\Lambda(\mu^s_t)\cdot\partial_s\nabla\psi^s_t} 
  + \ip{\nabla\psi_t^s,\partial_s\Lambda(\mu^s_t)\cdot\nabla\psi^s_t} \dx{t}
   \end{align*}
   From the continuity equation we infer that for any $\phi\in\R^\cX$
   \begin{align*}
     \ip{\phi,\partial_t\partial_s\mu^s_t} =\ip{\phi,\partial_s\partial_t\mu^s_t} = \ip{\nabla\phi,\Lambda(\mu^s_t)\cdot \partial_s\nabla\psi^s_t} + \ip{\nabla\phi,\partial_s\Lambda(\mu^s_t)\cdot\nabla\psi^s_t}\;.  
   \end{align*}
Plugging this into \eqref{eq:perturb1} for $s=0$ and integrating by parts in $t$ yields:
\begin{align*}
  0
   =
\int_0^1 2\, \ip{\partial_t\psi_t,\partial_s\vert_{s=0}\mu^s_t} + \ip{\nabla\psi_t,\partial_s\vert_{s=0}\Lambda(\mu^s_t)\cdot\nabla\psi_t}\dx{t} \;.
\end{align*}
The claim then follows by noting that
\begin{align*}
 \MoveEqLeft{\ip{\nabla\psi_t,\partial_s\vert_{s=0}\Lambda(\mu^s_t)\cdot\nabla\psi_t}}\\
 &=
 \sum_z \partial_s\vert_{s=0}\mu^s_t(z)  \frac{1}{2} \partial_{\mu(z)}\sum_{x,y} (\psi_t(x)-\psi_t(y))^2 \Lambda\bigl(\mu_t(x)Q(\mu_t;x,y),\mu_t(y)Q(\mu_t;y,x)\bigr)\;,
\end{align*}
and using that the perturbation $\partial_s\mu^s_t$ was arbitrary.
\end{proof}
In order to give convenient expressions for the first and second
variation of the free energy~$\cF$ along a geodesic, we introduce the
following notation.

We set 
\begin{equation}
  L_\mu\psi(x):= \sum_yQ(\mu;x,y)\psi_y\;, \qquad   \hat L_\mu \sigma(y):= \sum_x\sigma_xQ(\mu;x,y)\;,            \label{e:curv:Lap}
\end{equation}
and note that $\ip{L_\mu\psi,\sigma}=\ip{\hat L_\mu\sigma ,\psi}$, so
$\hat L_\mu$ is the adjoint of $L_\mu$.  The master equation
\eqref{e:master} then reads $\dot\mu_t=\hat L_{\mu_t}\mu_t$. Note further that we can write
\begin{align*}
  \langle\psi,L_{\mu}\phi\rangle=-\langle\nabla \psi,\bigl(Q(\mu)\pi(\mu)\bigr)\cdot\nabla\phi\rangle\;,
\end{align*}
where we set $\bigl(Q(\mu)\pi(\mu)\bigr)_{xy}=Q(\mu)_{xy}\pi(\mu)_x$, which
is symmetric in $x,y$.

\begin{lemma}[First variation of the free energy]\label{lem:ent-first-var}
 Let $(\mu,\psi)\in\CE$ be a solution to the continuity equation. Then it holds
\begin{equation}\label{e:first:variation:entropy}
 \frac{d}{dt}\cF(\mu_t) =- \langle \mu_t , L_{\mu_t}\psi\rangle\;.
\end{equation}
\end{lemma}

Note that when the curve is a solution to the gradient flow equation,
the right-hand side is indeed the discrete Fisher information, in
accordance with \eqref{e:FED}.

\begin{proof}
 Starting from the expression
 \begin{align*}
   \cF(\mu)= \sum_x\mu(x)\log\mu(x) + U(\mu)\;,
 \end{align*}
 recalling that $\partial_{\mu_x}U(\mu)=H_x(\mu)=-\log \pi_x(\mu)-\log Z(\mu)$, and
 setting $\rho_t=\mu_t/\pi(\mu_t)$, we obtain from the continuity equation
 \begin{equation*}\begin{aligned}%\label{eq:deriv-ent}
\frac{d}{dt} \cF(\mu_t)
  & = \sum_x \Bigl(1-\log Z(\mu_t)+\log \mu_t(x)-\log \pi_x(\mu)\Bigr)\partial_t\mu_t(x)
 = \ip{ \log \rho_t,\dot\mu_t}\\
& =\ip{\nabla\log \rho_t\ ,\ \Lambda(\mu_t)\cdot\nabla \psi_t}
 =  - \ip{\mu_t \ ,\ L_{\mu_t} \psi_t}\;.
\end{aligned}\end{equation*}
Here, we have also used in the last step that
\begin{align*}
  \Lambda(\mu_t)(x,y)= \frac{\nabla\rho_t(x,y)}{\nabla\log\rho_t(x,y)}Q(\mu_t;x,y)\pi(\mu_t)(x)\;,
\end{align*}
and integrated by parts.
\end{proof}

To give an expression of the second variation of $\cF$, we further
introduce the following notation.

Let $\partial_{\mu_z}Q(\mu;x,y)$ denote the partial derivative of $Q(\cdot;x,y)$ with respect to $\mu_z$. Then we write
\begin{align}
  DQ(\mu,\sigma;x,y):= \sum_{z\in\cX}\partial_{\mu_z}Q(\mu;x,y)\sigma_z\;.  \label{e:curv:DQ}
\end{align}
Furthermore, let us write 
\begin{align}
  M(\mu)\nabla\psi(x,y):=\sum_{z,w}M(\mu;z,w,x,y)\nabla\psi(z,w)\;, \notag\text{where} \\
  M(\mu;z,w,x,y):=\mu_x\Lambda(\mu)(z,w)\bigl[\partial_{\mu_z}Q(\mu;x,y)-\partial_{\mu_w}Q(\mu;x,y)\bigr]\;. \label{e:curv:M}
\end{align}
Then, we set
\begin{align*}
  \partial_i\Lambda(\mu)(x,y)&=\partial_i\Lambda\bigl(\mu_xQ(\mu;x,y),\mu_yQ(\mu;y,x)\bigr)\;, i=1,2\;,\\
 \hat L\Lambda(\mu)(x,y) &= \partial_1\Lambda(\mu)(x,y)\hat L_{\mu}\mu(x)Q(\mu;x,y) + \partial_2\Lambda(\mu)(x,y)\hat L_\mu\mu(y)Q(\mu;y,x)\;,\\
 R\Lambda(\mu)(x,y) &= \partial_1\Lambda(\mu)(x,y)\mu(x) DQ(\mu,\hat L_\mu\mu;x,y)+ \partial_2\Lambda(\mu)(x,y)\mu(y)DQ(\mu,\hat L_\mu\mu;y,x)\;.
\end{align*}
Finally, we can define the following quantity:
\begin{align}
  \cB(\mu,\psi) &:= 
  \frac12 \ip{\nabla\psi,\hat L\Lambda(\mu)\cdot\nabla\psi}
  -\ip{\nabla\psi,\Lambda(\mu)\cdot\nabla L_\mu\psi} \label{e:curv:B} \\
  &\qquad +\frac12 \ip{\nabla\psi,R\Lambda(\mu)\cdot\nabla\psi} 
  +\ip{\nabla\psi, M(\mu)\nabla\psi}\;, \notag
\end{align}

\begin{remark}
  Note that in the case of a linear Markov chain, the last two terms
  in the definition of $\cB$ vanish and we recover the formula of \cite{EM11} for the second derivative of the entropy along geodesics.
\end{remark}

\begin{lemma}[Second variation of the free energy]\label{lem:ent-second-var}
 Let $(\mu_t)_t$ be a $\cW$-geodesic contained in $\cP_*(\cX)$ and let $(\psi_t)$ be the unique potential such that $(\mu,\psi)\in\CE$. Then it holds
\begin{equation*}
 \pderiv[2]{}{t}\cF(\mu_t) = \Hess \cF(\mu_t)[\nabla\psi_t] = \cB(\mu_t,\psi_t)\;.
\end{equation*}
\end{lemma}
\begin{proof}
From ~\eqref{e:first:variation:entropy} we get
\[\begin{split}
  \pderiv[2]{}{t}\cF(\mu(t)) &=  
 -\pderiv{}{t}\ip{\hat L_{\mu_t}\mu_t,\psi_t}\\
  &=
 -\ip{\hat L_{\mu_t}\mu_t,\dot\psi_t} - \ip{\dot\mu_t,L_{\mu_t}\psi_t} - \ip{\mu_t,\bigl(\partial_tL_{\mu_t}\bigr)\psi_t}\\
 &=: I_1 +I_2+I_3\;.
\end{split}\]
To calculate $I_1$, first note that
\begin{align*}
  \partial_{\mu(z)}\Lambda(\mu)(x,y)&=
  \partial_1\Lambda(\mu)Q(\mu;x,y)\delta_{xz}
  +\partial_2\Lambda(\mu)Q(\mu;y,x)\delta_{yz}\\
  &\quad+\partial_1\Lambda(\mu)\mu(x)\partial_{\mu_z}Q(\mu;x,y)
+\partial_2\Lambda(\mu)\mu(y)\partial_{\mu_z}Q(\mu;y,x)\;,
\end{align*}
where $\delta_{xz}$ denotes the Kronecker delta.
Hence, we infer from the geodesic equation~\eqref{e:def:MF:geodesic:vectorfield} and \eqref{e:curv:DQ} that 
\begin{align*}
  I_1 = \frac12 \ip{\nabla\psi_t,\hat L\Lambda(\mu_t)\cdot\nabla\psi_t}
      + \frac12 \ip{\nabla\psi_t,R\Lambda(\mu_t)\cdot\nabla\psi_t}\;.
\end{align*}
The continuity equation $\dot\mu_t = -\nabla\cdot\bigl(\Lambda(\mu_t)\cdot\nabla\psi_t\bigr)$ readily yields that 
\begin{align*}
  I_2 = -\ip{\nabla\psi_t,\Lambda(\mu_t)\cdot\nabla L_{\mu_t}\psi_t}\;.
\end{align*}
To calculate $I_3$, note that for any $\phi$ we have
\begin{align*}
  \partial_tL_{\mu_t}\phi =  DQ\bigl(\mu_t, \dot\mu_t\bigr)\phi = - DQ\Bigl(\mu_t, \nabla\cdot\bigl(\Lambda(\mu_t)\cdot\nabla\psi_t\bigr)\Bigr)\phi\;,
\end{align*}
while for any $\mu$ and $\psi$ we have
\begin{align}\nonumber
   \MoveEqLeft{-\ip*{\mu,DQ\Bigl(\mu, \nabla\cdot\bigl(\Lambda(\mu)\cdot\nabla\psi\bigr)\Bigr)\psi}}\\\nonumber
~=~& \sum_{x,y,z}\mu_x\partial_{\mu_z}Q(\mu;x,y)\bigl[\nabla\cdot(\Lambda(\mu)\nabla\psi)\bigr](z)\nabla\psi(x,y)\\\nonumber
~=~&\frac12\sum_{x,y,z,w}\mu_x\partial_{\mu_z}Q(\mu;x,y)\Lambda(\mu)(z,w)\bigl[\nabla\psi(w,z)-\nabla\psi(z,w)\bigr]\nabla\psi(x,y)\\\nonumber
~=~&-\frac12\sum_{x,y,z,w}\nabla\psi(x,y)\nabla\psi(z,w)\mu_x\Lambda(\mu)(z,w)\bigl[\partial_{\mu_z}Q(\mu;x,y)-\partial_{\mu_w}Q(\mu;x,y)\bigr]\\\label{eq:Mcalc}
~=~&-\ip{\nabla\psi,M(\mu)\nabla\psi}\;.
\end{align}
Thus, we get $I_3=\ip{\nabla\psi_t,M(\mu_t)\nabla\psi_t}$. As
$I_1+I_2+I_3=\cB(\mu_t,\psi_t)$, this yields the claim.
\end{proof}

We can now state the following equivalent characterizations of lower Ricci bounds: 

\begin{theorem}\label{thm:Ric-equiv}
  Let $\kappa \in \R$. For a non-linear Markov triple $(\cX,Q,\pi)$
  the following assertions are equivalent:
\begin{enumerate}
\item $\Ric(\cX,Q,\pi) \geq \kappa$\;;
\item For all $\mu \in \cP_*(\cX)$ and $\psi \in \R^\cX$ we have
\begin{align*}
 \cB(\mu, \psi) \geq \kappa \cA(\mu, \psi)\;.
\end{align*}
\item The following \emph{Evolution
  Variational Inequality} EVI$_\kappa$ holds: for all $\mu, \nu \in \cP(\cX)$ and all $t \geq 0$:
\begin{align}\tag*{EVI$_\kappa$}\label{eq:EVI}
  \frac{1}{2}\frac{\mathup{d}^+}{\mathup{d}t} \cW(\mu_t, \nu)^2 + \frac{\kappa}{2}
\cW( \mu_t, \nu)^2 \leq \cF(\nu) - \cF(\mu_t)\;,
\end{align}
where $\mu_t$ denotes the solution to the non-linear Fokker--Planck equation starting from $\mu$, i.e.~$\dot\mu_t = \hat L_{\mu_t}\mu_t = \mu_tQ(\mu_t)$ and $\mu_0=\mu$;
%
% \item For all $\mu, \nu \in \cP_*(\cX)$, \eqref{eq:EVI} holds for
%   all $t \geq 0$;
%%
% \item For all $\mu \in \cP_*(\cX)$ we have 
% %
% \begin{align*}
% \Hess \cF(mu) \geq \kappa\;;
% \end{align*}
% %
% \item For all $\mu_0, \mu_1 \in \cP_*(\cX)$ there exists a
%   constant speed geodesic $(\mu_t)_{t \in [0,1]}$ satisfying
%   \begin{align*}
%     \cF(\mu_t)\leq (1-t)\cF(\mu_0)+t\cF(\mu_1) -\frac12t(1-t)\cW(\mu_0,\mu_1)^2\;.
%   \end{align*}
\end{enumerate}
\end{theorem}

By Lemma \ref{lem:ent-second-var}, (2) corresponds to a lower bound
$\kappa$ on the Hessian of $\cF$ in the Riemannian structure on
$\cP_*(\cX)$ induced by $\cW$. Note that the equivalence of (1) and
(2) is a non-trivial assertion, since the Riemannian metric
degenerates at the boundary of $\cP(\cX)$.

\begin{proof}
  The proof is based on an argument of Daneri and Savar\'e \cite{DS08}
  suitably adapted to the discrete setting. We can follow verbatim the
  proof of \cite[Thm.~4.5]{EM11} where the analogue of
  Thm.~\ref{thm:Ric-equiv} is proven for linear Markov chains. The
  core of the argument is a variation of the action along the
  evolution equation, \cite[Lem.~4.6]{EM11}. To accommodate the
  additional terms arising from the non-linear structure in the
  present situation, we have to replace that lemma with Lemma
  \ref{lem:dan-sav} below.
\end{proof}

\begin{lemma}\label{lem:dan-sav}
  Let $\{\mu^s\}_{s \in [0,1]}$ be a smooth curve in $\cP_*(\cX)$. For
  each $t \geq 0,$ let $\mu_t^s$ denote the solution of the non-linear Fokker--Planck equation at time $s+t$ starting from $\mu^s$ and let
  $\{\psi_t^s\}_{s \in [0,1]}$ be a smooth curve in $\R^\cX$ satisfying the
  continuity equation
\begin{align*}
  \partial_s \mu_t^s + \nabla \cdot ( \Lambda(\mu_t^s) \cdot \nabla \psi_t^s) =
  0\;, \qquad s \in [0,1]\;.
\end{align*}
Then the identity
\begin{align*}
 \frac12 \partial_t \cA(\mu_t^s, \psi_t^s)
 	 + \partial_s \cF(\mu_t^s) = - s \cB(\mu_t^s, \psi_t^s)
\end{align*}
holds for every $s \in [0,1]$ and $t \geq 0$.
\end{lemma}
\begin{proof}
First of all, setting $\rho^s_t=\frac{\mu^s_t}{\pi(\mu^s_t)}$ we compute as in Lemma \ref{lem:ent-first-var} that
\begin{align*}
\partial_s \cF(\mu_t^s)
  & = \ip{\log \rho^s_t\ ,\ \partial_s\mu^s_t}
 =  - \ip{\hat L_{\mu^s_t}\mu_t^s\ ,\ \psi^s_t} \;.
\end{align*}
Furthermore, 
\begin{align*}
 \frac12 \partial_t \cA(\mu_t^s, \psi_t^s)
&  =  \ip{\partial_t\nabla \psi_t^s\ ,\, \Lambda(\mu^s_t)\nabla \psi_t^s}
 + \frac12\ip{ \nabla \psi_t^s\ ,\, \partial_t\Lambda(\mu^s_t)
\cdot\nabla \psi_t^s }
 \\& =: \bar{I}_1 + \bar{I}_2\;.			
\end{align*}
In order to further manipulate $\bar{I}_1$ we first note that 
\begin{align*}
  \partial_t\mu^s_t=s\cdot\hat L_{\mu^s_t}\mu^s_t\;.
\end{align*}
Further, we observe that for any $\phi\in \R^\cX$ 
\begin{align}\nonumber
\MoveEqLeft{\ip{\nabla\phi,\Lambda(\mu^s_t)\cdot \partial_t\nabla\psi^s_t}
+
\ip{\nabla\phi, \partial_t\Lambda(\mu^s_t)\cdot \nabla\psi^s_t}} \\ \label{eq:partial-st}
~=~
&\ip{\mu^s_t,L_{\mu^s_t}\phi} + s\ip{\partial_s\mu^s_t,L_{\mu^s_t}\phi} + s\ip{\mu^s_t,\partial_sL_{\mu^s_t}\phi}\;.
\end{align}
To show \eqref{eq:partial-st}, note that the left-hand side equals
$\partial_t \partial_s \ip{\mu_t^s,\phi}$, while the right-hand side equals
$\partial_s \partial_t \ip{\mu_t^s,\phi}$. Integrating by parts repeatedly and using \eqref{eq:partial-st} we obtain
\begin{align*}
  \bar{I}_1 
  &  =  - \ip{\nabla \psi_t^s\ ,\, \partial_t\Lambda(\mu^s_t)\cdot\nabla\psi^s_t} + \ip{\mu^s_t,L_{\mu^s_t}\psi^s_t}
	    + s \ip{ \partial_s\mu_t^s,L_{\mu^s_t}\psi^s_t}
            +s \ip{\mu^s_t,(\partial_sL_{\mu^s_t})\psi^s_t}
\\  & = -2\bar{I}_2 - \partial_s \cF(\mu_t^s)
		 +  s  \ip{\nabla \psi_t^s\ ,\,
   			\Lambda(\mu^s_t)\cdot\nabla L_{\mu^s_t} \psi_t^s}  
	  + s\ip{\mu^s_t,(\partial_sL_{\mu^s_t})\psi^s_t}
\end{align*}
Thus, we arrive at
\begin{align*}
  \MoveEqLeft{\frac12 \partial_t \cA(\mu_t^s, \psi_t^s) + \partial_s \cF(\mu_t^s)}\\
= &-\frac12\ip{ \nabla \psi_t^s\ ,\, \partial_t\Lambda(\mu^s_t)
\cdot\nabla \psi_t^s }
 +  s  \ip{\nabla \psi_t^s\ ,\,
   			\Lambda(\mu^s_t)\cdot\nabla L_{\mu^s_t} \psi_t^s}  
	  + s\ip{\mu^s_t,(\partial_sL_{\mu^s_t})\psi^s_t}\;.
\end{align*}
To conclude, it suffices to note that 
\begin{align*}
  \partial_t\Lambda(\mu^s_t)&=s\cdot\hat L\Lambda(\mu^s_t) +s\cdot R\Lambda(\mu^s_t)\;,
\end{align*}
further remark that for any $\phi$ we have
\begin{align*}
  \partial_sL_{\mu^s_t}\phi = DQ(\mu^s_t,\partial_s\mu^s_t)\phi= - DQ\Bigl(\mu^s_t, \nabla\cdot\bigl(\Lambda(\mu^s_t)\cdot\nabla\psi^s_t\bigr)\Bigr)\phi\;,
\end{align*}
and then use again \eqref{eq:Mcalc}.
\end{proof}

To end this section, we use Theorem~\ref{thm:Ric-equiv} to give an
expression of the optimal lower Ricci bound on the two point space. 
\begin{lemma}[Two-point space]\label{lem:kappa:two-point}
  Let $\bigl(\set{0,1}, Q,\pi\bigr)$ be a non-linear Markov triple
  on the base space $\cX=\set{0,1}$ and let $p(\mu) := Q(\mu;0,1)$ and
  $q(\mu):= Q(\mu;1,0)$ as well as $p'(\mu)=[\partial_{\mu_0}-\partial_{\mu_1}] p(\mu)$ and $q'(\mu)=[\partial_{\mu_1}-\partial_{\mu_0}] q(\mu)$. Then, the optimal constant $\kappa$ such that
  $\Ric(\set{0,1},Q,\pi)\geq \kappa$ is given by
\begin{equation}\label{e:kappa:two-point}
\begin{split}
  \kappa_{\text{opt}} = \inf_{\mu \in \cP(\cX)} \Biggl( \frac{p(\mu)+q(\mu)}{2} &+ \frac{\mu(0) p'(\mu) + \mu(1) q'(\mu)}{2} \\
   &+ \frac{\Lambda(\mu)(0,1)}{2} \bra*{\frac{1}{\mu(0)\,\mu(1)} + \frac{p'(\mu)}{p(\mu)} + \frac{q'(\mu)}{q(\mu)}} \Biggr)
\end{split}
\end{equation}
\end{lemma}

\begin{remark}
  Note that in the case of a linear Markov chain, where $p$ and $q$
  are independent of $\mu$, and in particular $p'\equiv 0 \equiv q'$,
  we recover the formula in~\cite[Remark 2.11]{Maas2011}.
\end{remark}

\begin{proof}
First, we compute from~\eqref{e:curv:B} for any $\mu\in\cP_*(\set{0,1})$ and non-constant $\psi$:
\[
 \begin{split}
 \frac{\cB(\mu,\psi)}{(\psi(0)-\psi(1))^2} = &\frac{1}{2}\Bigl(\partial_1\Lambda(\mu)(0,1)p(\mu)\hat L_\mu\mu(0)
  +\partial_2\Lambda(\mu)(0,1)q(\mu)\hat L_\mu\mu(1) \Bigr)\\
 &+\Lambda(\mu)(0,1) \bigl(p(\mu) +q(\mu)\bigr)  \\
 &+\frac{1}{2}\partial_1\Lambda(\mu)(0,1)\mu(0)\bra*{ \partial_{\mu_0}p(\mu) \hat L_\mu\mu(0) +\partial_{\mu_1}p(\mu)\hat L_\mu\mu(1)}  \\
 &+\frac{1}{2}\partial_2\Lambda(\mu)(0,1)\mu(1)\bra*{ \partial_{\mu_0}q(\mu) \hat L_\mu\mu(0) +\partial_{\mu_1}q(\mu) \hat L_\mu\mu(1)} \\
 &+\Lambda(\mu)(0,1) \bigl[\mu(0)\bigl(\partial_{\mu_0}p(\mu)-\partial_{\mu_1}p(\mu)\bigr)-\mu(1)\bigl(\partial_{\mu_0}q(\mu)-\partial_{\mu_1}q(\mu)\bigr)\bigr]\;.
\end{split}
\]
Now, note that $\hat L_\mu\mu(0)=-\hat L_\mu\mu(1)= \mu(1) q(\mu)-\mu(0) p(\mu)$, yielding
\[
\begin{split}
\MoveEqLeft{\frac{\cB(\mu,\psi)}{(\psi(0)-\psi(1))^2}} \\ 
&= \Lambda(\mu)(0,1) \bigl[(p(\mu) +q(\mu) + \mu(0)p'(\mu)+\mu(1) q'(\mu)\bigr] \\
 &\phantom{=}+ \frac{1}{2}\bigl[\partial_1\Lambda(\mu)(0,1)p(\mu)  - \partial_2\Lambda(\mu)(0,1)q(\mu)\bigr]\bra*{\mu(1)q(\mu)-\mu(0)p(\mu)} \\
 &\phantom{=}+ \frac12\bigl[\partial_1\Lambda(\mu)(0,1)\mu(0)p'(\mu) - \partial_2\Lambda(\mu)(0,1)\mu(1) q'(\mu)\bigr]\bra*{\mu(1)q(\mu)-\mu(0)p(\mu)} \\
\end{split}
\]
Furthermore, $\cA(\mu,\psi) = \Lambda(\mu)(0,1) (\psi(1)-\psi(0))^2$. Thus by Theorem \ref{thm:Ric-equiv} we get the optimal curvature bound $\kappa_{\text{opt}}$ by dividing the above identity by $\Lambda(\mu)(0,1)$ and minimize in $\mu$. Now, we use the identities
\[
  \Lambda_1(a,b)(a-b) = \Lambda(a,b)- \frac{\Lambda(a,b)^2}{a} \quad\text{and}\quad   \Lambda_2(a,b)(a-b) = -\Lambda(a,b)+ \frac{\Lambda(a,b)^2}{b}
\]
to get rid of the partial derivatives and obtain after some further simplifications the result~\eqref{e:kappa:two-point}.
\end{proof}

\section{Consequences of Ricci bounds}
\label{sec:consequ}

In this section we derive consequences of Ricci curvature lower bounds
for non-linear Markov chains in terms of functional inequalities and
the trend to equilibrium for the dynamics. Throughout this section,
let $(\cX,Q,\pi)$ be a non-linear Markov triple satisfying Assumption
\ref{ass:GibbsPotential}.

We first note the following expansion bound for the transport
distance between solutions to the non-linear Markov dynamics.
\begin{proposition}
 Assume that $\Ric(\cX,Q)\geq \kappa$ for some $\kappa \in\R$. Then for any two solutions $(\mu_t^i)_{t\geq0}$ to the non-linear evolution equation $\dot\mu^i_t=\mu^i_tQ(\mu^i_t)$, $i=1,2$ we have
 \begin{align*}
   \cW(\mu^1_t,\mu^2_t)\leq e^{-\kappa t}\cW(\mu^1_0,\mu^2_0)\;.
 \end{align*}
\end{proposition}
In particular, when $\kappa > 0$, solutions with different initial data get closer at an exponential speed. 
\begin{proof}
  This is a consequence of the~\ref{eq:EVI}. It follows from \cite[Prop.~3.1]{DS08} applied to the functional $\cF$ on the metric space $(\cP(\cX),\cW)$.
\end{proof}
Next we prove some consequences of Ricci bounds in terms of different
functional inequalities. These results can be seen as non-linear
discrete analogues of classical results of Bakry and \'Emery \cite{BE85} and of
Otto and Villani \cite{OV00}. They extend results that have been
obtained in \cite{EM11} for linear Markov chains, and are reminiscent
of results of Carrillo, McCann and Villani \cite{Carrillo2003} obtained for
McKean--Vlasov equations in a continuous setting.

Let $\cF$ be the free energy functional associated with $(\cX,Q,\pi)$
given by %\eqref{e:def:FreeEnergy}
\begin{equation*}
 \cF(\mu)=\sum_{x\in\mathcal{X}}\mu_x \log\mu_x + U(\mu), \qquad\text{with}\qquad U(\mu)=\sum_{z\in\mathcal{X}}\mu_z\, K_z(\mu),
\end{equation*}
 and recall that $\cF$ attains its minimum on $\cP(\cX)$. We set
 \begin{align*}
   \cF_*(\mu):=\cF(\mu)-\min\limits_{\nu\in\cP(\cX)}\cF(\nu)\;.
 \end{align*}
 so that $\min \cF_* = 0$. Recall that $\cI$ is the discrete Fisher information, given by
% \eqref{e:def:dissipation}
 \begin{align*}%\label{e:def:dissipation}
    \mathcal{I}(\mu) =\frac{1}{2} \sum\limits_{x,y\in\cX} \Theta\bra*{\mu_{x}Q_{xy}(\mu), \mu_{y}Q_{yx}(\mu)} \;,\qquad  \Theta(a,b) = (a-b)(\log a - \log b)\;, 
  \end{align*}
  provided $\mu\in\cP_*(\cX)$ and $\cI(\mu)=+\infty$ else. Recall
  that $\cI$ gives the dissipation of $\cF$ along a solution $(\mu_t)$
  to the non-linear Fokker--Planck equation
  $\dot \mu_t=\mu_tQ(\mu_t)$.  More precisely, we have
\begin{align*}
  \pderiv{}{t} \cF(\mu_t) &= -\cI(\mu_t)\;.
\end{align*}
 Note further that with $\rho=\mu/\pi(\mu)$ we have the expression $\cI(\mu)=\cA(\mu,-\log\rho)$.
The next result relates $\cF$, $\cI$ and the transport distance $\cW$
under a Ricci bound.
\begin{theorem}\label{thm:FWI}
  Assume that $\Ric(\cX,Q,\pi)\geq\kappa$ for some $\kappa \in \R$. Then the $\cF\cW\cI$ inequality holds with constant $\kappa\in\R$, i.e.~for all $\mu,\nu\in \cP(\cX)$,
    \begin{align*}\tag*{$\cF\cW\cI(\kappa)$}\label{eq:FWI}
      \cF(\mu)\leq \cF(\nu) +\cW(\mu,\nu)\sqrt{\cI(\mu)} -\frac{\kappa}{2}\cW(\mu,\nu)^2\;.
    \end{align*}
\end{theorem}
\begin{proof}
  Fix $\mu,\nu\in\cP(\cX)$ and assume without restriction that
  $\mu\in\cP_*(\cX)$ since otherwise there is nothing to prove. Denote
  by $\mu_t$ the solution to $\dot\mu_t=\mu_tQ(\mu_t)$ with
  $\mu_0=\mu$ and set $\rho_t=\mu_t/\pi(\mu_t)$. Theorem
  \ref{thm:Ric-equiv} yields that~\ref{eq:EVI} holds, so in
  particular for $t=0$:
  \begin{align*}
    \cF(\mu)\leq \cF(\nu)-\frac12\frac{\mathup{d}^+}{\mathup{d}t}\Big|_{t=0}\cW(\mu_t,\nu)^2 - \frac{\kappa}{2}\cW(\mu,\nu)^2\;.
  \end{align*}
 From the triangle inequality and the fact that $t\mapsto \mu_t$ is continuous with respect to $\cW$ we obtain
 \begin{align*}
   -\frac12\frac{\mathup{d}^+}{\mathup{d}t}\Big|_{t=0}\cW(\mu_t,\nu)^2 &= \liminf_{s\searrow0}\frac{1}{2s}\Bigl(\cW(\mu,\nu)^2-\cW(\mu_s,\nu)^2\Bigr)\\
 &\leq 
\limsup_{s\searrow0}\frac{1}{s}\cW(\mu_s,\mu)\cdot\cW(\mu,\nu)\;. 
 \end{align*}
Now, note that since $(\mu_t,-\log\rho_t)\in\CE$ we can estimate
\begin{align*}
  \cW(\mu_s,\mu)\leq \int_0^s\sqrt{\cA(\mu_r,\log\rho_r)}\ dr=\int_0^s\sqrt{\cI(\mu_r)}\ dr\;.
\end{align*}
Since $t\mapsto\cI(\mu_t)$ is a continuous function, we obtain
\begin{align}\label{eq:deriv-W}
  \limsup_{s\searrow0}\frac{1}{s}\cW(\mu_s,\mu)\leq \sqrt{\cI(\mu)}\;,
\end{align}
which yields the claim.
\end{proof}

\begin{theorem}
 Assume that $\Ric(\cX,Q,\pi)\geq \lambda$ for some $\lambda>0$. Then the following hold:
 \begin{itemize}
 \item[(i)] there exists a unique stationary point $\pi^*$, it is the unique minimizer of $\cF$;
 \item[(ii)] the \emph{modified logarithmic Sobolev inequality} with constant $\lambda>0$ holds, i.e.~for all $\mu\in \cP(\cX)$,
    \begin{align}\tag*{$\MLSI(\lambda)$}\label{eq:MLSI}
      \cF_*(\mu)\leq \frac{1}{2\lambda}\cI(\mu)\;;
    \end{align}
\item[(iii)] for any solution $(\mu_t)_{t\geq0}$ to $\dot\mu_t=\mu_tQ(\mu_t)$ we have exponential decay of the free energy:
  \begin{align}\label{eq:exp-ent}
    \cF_*(\mu_t)\leq e^{-2\lambda t} \cF_*(\mu_0)\;;
  \end{align}
 \item[(iv)] the \emph{transport-entropy inequality} with constant $\lambda>0$ holds, i.e.~for all $\mu\in\cP(\cX)$,
    \begin{align}\tag*{ET$(\lambda)$}\label{eq:ET}
      \cW(\mu,\pi_*) \leq \sqrt{\frac{2}{\lambda}\cF_*(\mu)}\;.
    \end{align}
  % \item[(3)] a \emph{Poincar\'e} inequality if for all $\mu\in\cP(\cX)$ and $\phi\in\R^\cX$ such that $\sum_x\phi(x)\pi(\mu;x)=0$, 
  %   \begin{align}\tag{P$(\lambda)$}\label{eq:P}
  %     \sum_x|\phi(x)|^2\pi(\mu;x) \leq \frac{1}{\lambda} \frac12\sum_{x,y}|\nabla\phi(x,y)|^2Q(\mu;x,y)\pi(\mu;x)\;.
  %   \end{align}
 \end{itemize}
\end{theorem}

\begin{proof}
  {\bf (i)} From Proposition \ref{prop:charact:stationary} we know that the set $\Pi_*$
  of stationary points is non-empty and that it coincides with the set
  of local minimizers of $\cF$. Assume by contradiction that $F$ has two distinct local minima
  at points $\mu_{0}$ and $\mu_{1},$ with
  $F(\mu_{0})\leq F(\mu_{1})$ and let $(\mu_s)_{s\in[0,1]}$ be a constant speed geodesic connecting $\mu_{0}$ and $\mu_{1}$. Then we infer from $\Ric(\cX,Q,\pi)\geq \lambda$ that 
  \[\begin{split}
    \cF(\mu_s)&\leq (1-s)\cF(\mu_{0})+ s\cF(\mu_{1})-\frac{\lambda}{2}s(1-s)\cW(\mu_{0},\mu_{1})^2\;.
    \end{split}
  \]
  Since $\mu_{1}$ is a local minimum, there is an $\epsilon>0$ such that $\cF(\mu_{1-\epsilon})\geq \cF(\mu_{1})$.
  This leads to
  \[
    \cF(\mu_{1})\leq \cF(\mu_{1-\epsilon})\leq \epsilon \cF(\mu_{0})+ (1-\epsilon)\cF(\mu_{1})-\frac{\lambda}{2}\epsilon(1-\epsilon)<\cF(\mu_{1})\;,
  \]
  a contradiction. Hence, $\Pi_*=\{\pi_*\}$ is a singleton and $\pi_*$
  is the unique global minimizer of $\cF$.  \medskip
  
  {\bf (ii)} By Theorem \ref{thm:FWI}, we have that~\ref{eq:FWI}
  holds. Applying~\ref{eq:FWI} with $\mu\in\cP(\cX)$ and $\nu=\pi_*$, noting that $\cF_*(\pi_*)=0$, and using Young's inequality
  \begin{align*}
    xy\leq cx^2 +\frac1{4c}y^2\qquad \forall x,y\in \R,\ c>0,
  \end{align*}
  with $x=\cW(\mu,\pi_*)$, $y=\sqrt{\cI(\mu)}$ and $c=\lambda/2$
  yields the claim.  \medskip

{\bf (iii)} 
  From~\ref{eq:MLSI} we infer that for a solution $(\mu_t)_t$ we have
  \begin{align*}
    \pderiv{}{t}\cF_*(\mu_t)= -\cI(\mu_t) \leq -2\lambda \cF_*(\mu_t)\;,
  \end{align*}
  and we obtain \eqref{eq:exp-ent} as a consequence of Gronwall's
  lemma.
 \medskip

 {\bf (iv)} It suffices to establish~\ref{eq:ET} for any
 $\mu\in\cP_*(\cX)$. The inequality for general $\mu$ can then be
 obtained by approximation, taking into account the continuity of
 $\cW$ with respect to the Euclidean metric on $\cP(\cX)$. So fix
 $\mu\in\cP_*(\cX)$, and let $\mu_t$ be the solution to the non-linear
 Fokker--Planck equation starting from $\mu$.  From Proposition
 \ref{prop:qualitative:longtime} we have that $\mu_t\to\pi_*$ as
 $t\to\infty$ and that
 \begin{align}\label{eq:asymptotic}
    \cF_*(\mu_t)\to 0 \quad\mbox{ and }\quad \cW(\mu,\mu_t)\to
    \cW(\mu,\pi_*)\;.
  \end{align}
  The last property follows from the continuity of $\cW$ with respect
  to the Euclidean distance. We now define the function
  $G:\R_+\to\R_+$ by
  \begin{align*}
    G(t)~:=~\cW(\mu,\mu_t) + \sqrt{\frac{2}{\lambda}\cF_*(\mu_t)}\;.
  \end{align*}
  Obviously we have $G(0)=\sqrt{\frac{2}{\lambda}\cF_*(\mu)}$ and by
  \eqref{eq:asymptotic} we have that $G(t)\to\cW(\mu,\pi_*)$ as
  $t\to\infty$. Hence it is sufficient to show that $G$ is
  non-increasing. To this end we show that its upper right derivative
  is non-positive. If $\mu_t\neq \pi_*$ we deduce from \eqref{eq:deriv-W} that
  \begin{align*}
    \frac{\mathup{d}^+}{\mathup{d}t} G(t)~\leq~\sqrt{\cI(\mu_t)} - \frac{\cI(\mu_t)}{\sqrt{2\lambda\cF_*(\mu_t)}}~\leq~0\ ,
  \end{align*}
  where we used~\ref{eq:MLSI} in the last inequality. If
  $\mu_t=\pi_*$, then the relation also holds true, since this implies
  that $\mu_r=\pi_*$ for all $r\geq t$.
\end{proof}

\section{Some examples of curvature bounds}\label{s:bounds}

We shall now compute lower bounds on the curvature for several examples of mean-field dynamics, inspired by classical models of statistical physics. 

\subsection{Curie-Weiss model}\label{sec:CW}

Let us consider the following example also mentioned in \cite[Example
4.2]{BDFR15b}, which is the infinite particle limit of the classical
Curie-Weiss model, one of the simplest examples of Markovian dynamic
exhibiting a phase transition. Let us take $\mathcal{X}=\{0,1\}$ and
define $K$ for $\beta>0$ by
\[
  K_x(\nu)=V_x+\beta\sum_{y\in\mathcal{X}} W_{xy}\,  \nu_{y}, \qquad  (x,\nu)\in\cX\times\cP(\cX),
\]
with $V\equiv0,\, W(0,0)=0=W(1,1),$ and $W(0,1)=1=W(1,0)$. Hence, we have
\[
  U(\nu) = 2\beta \nu_0 \nu_1 \;. 
\]

The free energy $\cF(\mu)$ for the Curie-Weiss model is given by
\begin{equation*}
\begin{split}
\sum_{x\in\mathcal{X}}\bra[\big]{\log\mu(x) + K(\mu;x)}\mu(x)&=\sum_{x\in\mathcal{X}}(\log\mu(x)  +\beta\sum_{y\in\mathcal{X}}W(x,y)\mu(y)\mu(x)\\
&=\mu(0)\log\mu(0)+\mu(1)\log\mu(1)+2\beta\mu(0)\mu(1).
\end{split}
\end{equation*}
Since $\mu(0)+\mu(1)=1$, we have that the free energy is essentially
given by the function $f_\beta: [0,1]\to \R$
\begin{equation*}
 f_\beta(u) = u \log u + (1-u) \log (1-u) + 2 \beta u (1-u) \quad\text{and}\quad f''_\beta(u) = \frac{1}{u(1-u)} - 4 \beta,
\end{equation*}
Hence, $f_\beta$ is convex on $[0,1]$ for $\beta \in [0,1]$ and non-convex for $\beta>1$.

The local detailed balance state $\pi(\mu)$~\eqref{e:def:piH} is given by
\begin{equation*}
  \pi_x(\mu) = \frac{1}{Z(\mu)} \exp\bra*{- H_x(\mu)} =\frac{1}{Z(\mu)} \exp\bra*{-2 \beta\sum_{y\in\mathcal{X}}W(x,y)\mu(y) } \ . 
\end{equation*}
Therefore, it holds 
\begin{align*}
  \pi_0(\mu)&=\frac{\exp\bra*{-2 \beta\mu(1) }}{\exp\bra*{-2 \beta \mu(0) }+\exp\bra*{-2 \beta \mu(1) }} \;, \\
  \pi_1(\mu)&=\frac{\exp\bra*{-2 \beta\mu(0) }}{\exp\bra*{-2 \beta \mu(0) }+\exp\bra*{-2 \beta \mu(1) }} \;. 
\end{align*}
We use Glauber rates and set 
\begin{align*}
 p(\mu) &= Q(\mu;0,1) = \sqrt{\frac{\pi_1(\mu)}{\pi_0(\mu)}} = \exp\bra[\big]{-\beta(\mu(0) - \mu(1))}  \;, \\
 q(\mu) &= Q(\mu;1,0) = \sqrt{\frac{\pi_0(\mu)}{\pi_1(\mu)}} = \exp\bra[\big]{\beta(\mu(0) - \mu(1))} = \frac{1}{p(\mu)} \;. 
\end{align*}
With this choice, we can estimate the Ricci curvature of the limit with the help of Lemma~\ref{lem:kappa:two-point}.
\begin{proposition}[$\lambda$-Convexity of Curie-Weiss model with Glauber rates]
  It holds for $\beta \in [0,1]$
  \begin{equation}\label{e:CW:kappaGlauber}
    \kappa_{\Glauber} = 2 (1 - \beta).
  \end{equation}
\end{proposition}
As a consequence of this curvature bound, one can derive the modified
logarithmic Sobolev inequality for the nonlinear dynamic. This
inequality could also be derived from a logarithmic Sobolev inequality
for the particle Gibbs sampler of \cite{Marton2015} and passing to the
limit in the number of particles. In \cite{Kraaij2016}, the mLSI was
also derived via convexity of the entropy, but along a different
family of interpolations of probability measures. At a technical
level, the proof of \cite{Kraaij2016} requires differentiating the
entropy three times rather than two, which involves more technical
estimates (this is not much of an issue for a two-point space system
like Curie-Weiss, but gets much more complicated for more involved
systems, as the ones we shall see later in this section).
\begin{proof}
  We set $\mu(0)=u$ and $\mu(1)=1-u$, for which the rates become
  $p(\mu)=\exp(-\beta(2u-1))=1/q(u)$.  First, we note that with the
  notation of Lemma~\ref{lem:kappa:two-point}, we have
  $p'(\mu)=-2\beta p(\mu)$ and $q'(\mu)=-2\beta q(\mu)$.
  
  The expression in the infimum of~\eqref{e:kappa:two-point} to optimize becomes  
\begin{equation*}
  \kappa(u) = \frac{p+q}{2} - \beta\bra*{u p  +(1-u) q} + \frac{\Lambda\bra*{up,(1-u)q}}{2}\bra*{ \frac{1}{u(1-u)}  - 4 \beta} .
\end{equation*}
It will be convenient to do the variable substitution $x= 2u-1$. For
obtaining the expression in a compact manner, we introduce two
auxiliary functions
\begin{align*}
    g_1(x) := \cosh(\beta x) - x \sinh(\beta x) \qquad\text{and}\qquad g_2(x) := \sinh(\beta x) - x \cosh(\beta x) .
\end{align*}
We then obtain, using the identities $up+(1-u)q=g_1(x)$,
$up-(1-u)q=-g_2(x)$ and
$\arctanh(x) = \frac{1}{2} \log\frac{1+x}{1-x}$ and after some
rewriting,
\begin{equation}\label{e:CW:p1}
  \kappa(x) = \cosh(\beta x ) -  \beta g_1(x) +  \frac{g_2(x)}{\beta x - \arctanh(x)}\bra*{ \frac{1}{1-x^2} - \beta}  .
\end{equation}
A simple evaluation yields $\kappa(0) = 2(1-\beta)$, where we note
$\frac{g_2(x)}{\beta x - \arctanh(x)} \to 1$ as $x\to 0$.

For the lower bound, we proceed in several steps, we first observe
that
\begin{align*}
  \cosh(\beta x ) - \beta g_1(x) = (1-\beta) \cosh(\beta x) + \beta x
  \sinh(\beta x) \geq 1 - \beta \;.
\end{align*}
Now, the claim follows once we have shown
\begin{equation}\label{e:CW:p2}
  \frac{g_2(x)}{\beta x - \arctanh(x)} \geq 1-x^2 \;.
\end{equation}
Indeed, the last term in~\eqref{e:CW:p1}, combined with the above
estimate, is bounded from below by $1-(1-x^2)\beta \geq 1-\beta$,
which proves~\eqref{e:CW:kappaGlauber}. To prove~\eqref{e:CW:p2}, we
do another substitution and set $x=\tanh(y)$. Therefore, the function
$g_2$ becomes after transformations by hyperbolic trigonometric
identities
  \begin{equation*}
     g_2(\tanh y) = \frac{\sinh(\beta \tanh(y)-y)}{\cosh y} .
  \end{equation*}
  and we can estimate the left hand side of~\eqref{e:CW:p2} by
  \[
     \frac{g_2(x)}{\beta x - \arctanh(x)} = \frac{1}{\cosh(y)} \frac{\sinh\bra*{\beta \tanh(y)-y}}{\beta \tanh(y)-y} \geq \frac{1}{\cosh(y)} \;,
  \]
  where we used the bound $\sinh(z)/z \geq 1$ for all $z\in \R$. This can be further estimated by observing that $\cosh(y)\geq 1$ for all $y\in \R$ and by using the identity $1-\tanh^2 y = 1/\cosh^2 y$, to obtain
  \[
    \frac{1}{\cosh(y)} \geq \frac{1}{\cosh^2(y)} = 1-\tanh^2(y) = 1 -x^2 \;,
  \]
  by the substitution $x=\tanh(y)$, which proves~\eqref{e:CW:p2}.   
\end{proof}
\subsection{General zero-range/misanthrope processes}
In this section, we consider mean-field limits of particle systems with rate matrix of the form
\begin{equation*}
Q(\mu; x,y) = p(x,y) \, c(\mu_x,\mu_y)
\end{equation*}
These systems generalize usual linear Markov chains encoded in $p(x,y)$ by an additional dependency of the jump rate on the population density of the departure and arrival site of the jump. 
This model, first introduced in~\cite{CT85},  incorporates many examples, such as for instance the zero range
process, for which $c(\mu_x,\mu_y) = b(\mu_x)$, but also interacting
agent/voter models~\cite{Villemonais2017}, for which
$c(\mu_x,\mu_y) = a(\mu_y)$.

Since our method in this section is perturbative in nature, we restrict to the complete
graph as underlying graph, that is $p(x,y) =1$ for all $x\ne y$. In this case the mean-field limit from the $N$-particle system was
derived in~\cite{GJ18} and the limit equation was investigated in~\cite{Sch19}.
Since positive curvature is know in the
case of independent particles on the complete graph \cite{EM11}, we expect that for
$c(\mu_x,\mu_y) = T + \tilde c(\mu_x,\mu_y)$ with bounded
$\tilde c: \cP(\cX) \times \cP(\cX) \to [0,\infty)$, we should also obtain
positive entropic curvature for the nonlinear models when $T$ is
sufficiently large.

To have a gradient flow formulation, the chain has to satisfy the local detailed balance
condition~\eqref{e:DBC}
\begin{equation}\label{e:ZRP:lDBC}
  \pi_x(\mu) \, c(\mu_x,\mu_y) = \pi_y(\mu) \, c(\mu_y, \mu_x)  \;.
\end{equation}
For the further analysis, we will focus on the separable case, where for some $a,b:[0,1]\to \R_+$ holds
\begin{equation*}
  c(\mu_x,\mu_y) = b(\mu_x) \, a(\mu_y) \;.
\end{equation*}
It is easy to verify that~\eqref{e:ZRP:lDBC} is satisfied for 
\begin{equation*}
  \pi_x(\mu) = \frac{1}{Z(\mu)} \frac{a(\mu_x)}{b(\mu_x)} \quad\text{and}\quad Z(\mu) =\sum_{x\in \cX} \frac{a(\mu_x)}{b(\mu_x)} \;.
\end{equation*}
This is of the form~\eqref{e:def:potentialU} for a potential $U$ given e.g.~by
\begin{equation*}
  U(\mu) = \sum_{x\in \cX} u(\mu_x) \qquad\text{with}\qquad u(r) = \int_r^1 \log \bra*{\frac{a(s)}{b(s)}} \dx{s}\;, 
\end{equation*}
i.e.~for $K$ given by $K_x(\mu)=u(\mu_x)/\mu_x$.
\begin{Example}
There are two subclasses of models of particular interest:
\begin{align*}
  Q_{xy}(\mu) = a(\mu_y) \qquad\text{and}\qquad  Q_{xy}(\mu) = b(\mu_x) \;.
\end{align*}
Both models satisfy the local detailed balance condition~\eqref{e:DBC} for
\[
  \pi^a_x(\mu) = \frac{a(\mu_x)}{Z^a(\mu)} \qquad\text{with}\qquad Z^a(\mu) = \sum_x a(\mu_x)
\]
and 
\[
  \pi^b_x(\mu) = \frac{1}{Z^b(\mu) b(\mu_x)} \qquad\text{with}\qquad Z^b(\mu) = \sum_x \frac{1}{b(\mu_x)}
\]
For the first, the interacting agent model from~\cite{Villemonais2017}
is recovered by setting $a(\mu_y) = T/d + f(\mu_y)$, where
$d =\abs{\cX}$ is the (constant) degree of the complete graph. For
these models,~\cite{Villemonais2017} proves a spectral gap via another
notion of discrete curvature, but which is not strong enough to derive
the mLSI. In the second case, this dynamic corresponds to (a scaling
limit of) a zero range-process. This type of particle system is
commonly used in statistical physics as a toy model for understanding
various large-scale features of interacting systems (scaling limits,
long-time behavior, phase transitions). We refer to \cite{Liggett2004}
for an overview. Long-time behavior of the $N$-particle system in
various situations was studied for example in \cite{Caputo2007,
  Caputo2009, FM16, Hermon2018}. Recently, Hermon and Salez
\cite{Hermon2019} significantly improved on the state of the art using
a combination of the Lu-Yau martingale method and a monotone coupling
argument, establishing a modified logarithmic Sobolev inequality
independent of the number of particles for mean-field zero-range
processes in a non-perturbative setting, even in some inhomogeneous
situations where the curvature approach \emph{cannot} work.
\end{Example}
In this separable case, we can prove the following statement.
\begin{theorem}[Curvature for separable kernels]
   Assume the rates are separable, given by 
   \begin{equation*}
     Q(\mu;x,y) = b(\mu_x) a(\mu_y) \ .
   \end{equation*}
   Suppose that 
   \begin{equation*}
     0<\underline{a} \leq a(\cdot) \leq \bar a \qquad\text{and}\qquad 0<\underline{b} \leq b(\cdot) \leq \bar b \ .
   \end{equation*}
   Moreover, assume that
   \begin{equation}\label{e:ZRP:lambda}
     \lambda := \frac{2 \, \overline{a} \Lip b + \overline{a} \Lip a}{2 \, \underline{a} \, \underline{b}} \leq 1  \ .
   \end{equation}
   Then $\Ric \geq \kappa$ in the sense of Theorem~\ref{thm:Ric-equiv} with $\kappa$ given by
   \begin{equation}\label{e:ZRP:kappa}
     \kappa := d\pra*{  \underline{a}\, \underline{b} - (1+\lambda)\frac{\bar{a}\, \bar{b}}{2} -  \frac{\bar a}{2} \bra*{ 2 +  \frac{ \bar b}{\underline{b}}} \Lip b -\frac{\overline{b}}{2} \bra*{1 +  \frac{\bar a}{\underline{a}}} \Lip a }
   \end{equation}
   Especially, in the regime $\frac{\max\set*{\Lip a , \Lip b}}{\min\set{\underline{a}, \underline{b}}} =: \eta \ll 1$ it holds
   \begin{equation}\label{e:ZRP:kappa:asymptotic}
     \kappa \geq d \, \underline{a} \, \underline{b} \, \bra[\big]{ 1 +  O(\eta)}\;.
   \end{equation}  
\end{theorem}
\begin{proof}
First, we evaluate some of the quantities occurring in the derivation of the curvature estimate. Let us start with~\eqref{e:curv:Lap}, for which we have
\begin{equation}\label{e:ZRP:Lap}
  L_{\mu} \psi(x) = \sum_{y} \bra*{\psi_y -\psi_x} b(\mu_x) a(\mu_y) \; 
\end{equation}
and
\begin{equation*}
  \hat L_{\mu} \sigma(y) = \sum_{x} \bra[\big]{\sigma_x b(\mu_x) a(\mu_y) - \sigma_y b(\mu_y) a(\mu_x)} \;. 
\end{equation*}
The next quantity~\eqref{e:curv:DQ} becomes
\begin{equation}\label{e:ZRP:DQ}
  DQ(\mu,\sigma;x,y) = \sum_{z} \partial_{\mu_z} b(\mu_x) a(\mu_y) \sigma_z = b'(\mu_x) a(\mu_y) \sigma_x +  b(\mu_x) a'(\mu_y) \sigma_y \;. 
\end{equation}
The last quantity is~\eqref{e:curv:M}
\begin{align*}
  M(\mu;z,w,x,y) = \frac{1}{2}\mu_x \Lambda(\mu)(z,w) \Bigl( & \delta_{z,x} b'(\mu_x) a(\mu_y) + \delta_{z,y} b(\mu_x) a'(\mu_y) \notag \\
  & - \delta_{w,x} b'(\mu_x) a(\mu_y) - \delta_{w,y} b(\mu_x) a'(\mu_y)  \Bigr)  
\end{align*}
from which after symmetrization, we obtain the identity 
\begin{equation}\label{e:ZRP:M2}
  M(\mu) \nabla\psi(x,y) = \mu_x \sum_{z} \bra*{ \nabla\psi(x,z) \Lambda(\mu)(x,z) b'(\mu_x) a(\mu_y) + \nabla\psi(y,z) \Lambda(y,z) b(\mu_x) a'(\mu_y)}
\end{equation}
We will use the following identity for the logarithmic mean
\begin{equation}\label{e:logMean:homoIdent}
  s \partial_1\Lambda(s,t) + t \partial_2\Lambda(s,t) = \Lambda(s,t) \;.
\end{equation}
To compensate off-diagonal terms, we need the following estimate for the logarithmic mean~\cite[Lemma A.2]{FM16}
\begin{equation}\label{e:logMean:FM15}
  r\bra*{\partial_1\Lambda(s,t) + \partial_2\Lambda(s,t)} + \Lambda(s,t) \geq \Lambda(r,s) + \Lambda(r,t) \;. 
\end{equation}
The above basic identities shall be used to estimate the four terms in~\eqref{e:curv:B}, which we denote by $\tI, \tII, \tIII$ and $\tIV$ in this order of occurrence.

\medskip

\emph{First term $\tI$:} Let us start estimating the first term in~\eqref{e:curv:B} and use the identity~\eqref{e:logMean:homoIdent} 
\begin{align}
 \tI &= \frac{1}{4} \sum_{x,y} \abs{\nabla \psi(x,y)}^2 \Bigl(  \partial_1\Lambda(\mu)(x,y) \sum_{z} \bra[\big]{\mu_z b(\mu_z) a(\mu_x) - \mu_x b(\mu_x) a(\mu_z)} b(\mu_x) a(\mu_y) \notag\\
 &\qquad\qquad\qquad\qquad\qquad + \partial_2\Lambda(\mu)(x,y) \sum_{z} \bra[\big]{\mu_z b(\mu_z) a(\mu_y) - \mu_y b(\mu_y) a(\mu_z)} b(\mu_y) a(\mu_x) \Bigr) \notag\\
 &= \frac{1}{4} \sum_{x,y,z} \abs{\nabla \psi(x,y)}^2 \mu_z b(\mu_z) a(\mu_x) a(\mu_y) \Bigl(  \partial_1\Lambda(\mu)(x,y) b(\mu_x) + \partial_2\Lambda(\mu)(x,y) b(\mu_y) \Bigr) \notag\\
 &\quad - \frac{1}{4} \sum_{x,y} \abs{\nabla \psi(x,y)}^2 A(\mu) \Bigl(  \partial_1\Lambda(\mu)(x,y) \mu_x b(\mu_x)^2 a(\mu_y)+ \partial_2\Lambda(\mu)(x,y) \mu_y b(\mu_y)^2 a(\mu_x) \Bigr) \notag\\
 &\geq \tI_1 - \frac{\sup_{\mu,x}\set{ b(\mu_x) A(\mu)}}{2} \cA(\mu,\psi) \geq \tI_1 - \frac{d\, \bar{a}\, \bar{b}}{2} \, \cA(\mu,\psi) \label{e:ZRP:I}
\end{align}
where we introduced $A(\mu) = \sum_{z} a(\mu_z)$.  Although $\tI_1$ is
non-negative, we will keep it to compensate for terms from $\tII$ and
$\tIV$. To do so, we compactify notation further by introducing the
tilted measure
\begin{equation*}
  \bar\mu^{b,a}_x = \mu_x \frac{b(\mu_x)}{a(\mu_x)} \ .
\end{equation*}
With this definition and with the one-homogeneity of $\Lambda$, we can rewrite
\[
  \Lambda(\mu)(x,y) = \Lambda\bra[\big]{\mu_x b(\mu_x) a(\mu_y) , \mu_y b(\mu_y) a(\mu_x)} = \Lambda\bra*{\bar\mu^{b,a}_x , \bar\mu^{b,a}_y} a(\mu_x) a(\mu_y)
\]
and likewise for the zero-homogeneous derivatives $(i=1,2)$
\[
  \partial_i\Lambda(\mu)(x,y)=\partial_i\Lambda\bra[\big]{\mu_x b(\mu_x) a(\mu_y) , \mu_y b(\mu_y) a(\mu_x)} = \partial_i\Lambda\bra*{ \bar\mu^{b,a}_x , \bar\mu^{b,a}_y} \;.
\]
With this notation we want to employ the estimate~\eqref{e:logMean:FM15} in the form
\begin{align*}
  \bar\mu^{b,a}_z\bra*{ \partial_1\Lambda\bra*{ \bar\mu^{b,a}_x , \bar\mu^{b,a}_y}+ \partial_2\Lambda\bra*{ \bar\mu^{b,a}_x , \bar\mu^{b,a}_y}} &\geq \Lambda\bra*{\bar\mu^{b,a}_z , \bar\mu^{b,a}_x} + \Lambda\bra*{\bar\mu^{b,a}_z , \bar\mu^{b,a}_y} - \Lambda\bra*{\bar\mu^{b,a}_x , \bar\mu^{b,a}_y} \\
  &= \frac{\Lambda(\mu)(z,x)}{a(\mu_x) a(\mu_z)} + \frac{\Lambda(\mu)(z,y)}{a(\mu_y) a(\mu_z)}  - \frac{\Lambda(\mu)(x,y)}{a(\mu_x) a(\mu_y)}  \ .
\end{align*}
Now, we can bound $\tI_1$ from below by
\begin{align}
  \tI_1 &\geq \underline{b}\; \frac{1}{4} \sum_{x,y,z} \abs{\nabla \psi(x,y)}^2 a(\mu_x) a(\mu_y) a(\mu_z) \bar\mu^{b,a}_z\bra*{ \partial_1\Lambda\bra*{ \bar\mu^{b,a}_x , \bar\mu^{b,a}_y}+ \partial_2\Lambda\bra*{ \bar\mu^{b,a}_x , \bar\mu^{b,a}_y}} \notag\\
  &\geq  \underline{a}\, \underline{b}\; \frac{1}{4} \sum_{x,y,z} \abs{\nabla \psi(x,y)}^2\bra[\big]{ \Lambda(\mu)(z,x) + \Lambda(\mu)(z,y)} - \frac{d \, \bar{a}\, \underline{b}}{2}\, \cA(\mu,\psi) \notag\\
  &\geq \underline{a}\, \underline{b} \, \tI_{2} - \frac{ d \, \bar{a}\, \overline{b}}{2}\, \cA(\mu,\psi)\;. \label{e:ZRP:I2}
\end{align}
\medskip

\emph{Second term $\tII$:} Let us continue with the second term in~\eqref{e:curv:B} for which we use~\eqref{e:ZRP:Lap}, symmetrize the sum and obtain
\begin{align}
  \tII &= - \ip{ \nabla \psi, \Lambda(\mu) \nabla L_{\mu} \psi} \notag\\
 &= \sum_{x,y,z} \nabla\psi(x,y)\nabla\psi(x,z) \Lambda(\mu)(x,y) b(\mu_x)a(\mu_z) \notag\\
  &= \frac12\sum_{x,y,z} a(\mu_z)\Lambda(\mu)(x,y)\nabla\psi(x,y)\Bigl(\underbrace{\nabla\psi(x,z)}_{\mathclap{=\nabla\psi(x,y)+\nabla\psi(y,z)}}b(\mu_x) - \nabla\psi(y,z)b(\mu_y)\Bigr) \notag\\
  &=  \frac{1}{2} \sum_{x,y} \abs{\nabla\psi(x,y)}^2 \Lambda(\mu)(x,y) b(\mu_x) \sum_z a(\mu_z) \notag\\
  &\quad +  \frac{1}{2} \sum_{x,y,z} \nabla \psi(x,y) \nabla\psi(x,z) \Lambda(\mu)(x,y) a(\mu_z) \bra[\big]{ b(\mu_x) - b(\mu_y)} \notag\\
  &\geq \bra*{ d\,\underline{b}\,\underline{a} - \frac{d \, \bar a\, \Lip b}{2}}\cA(\mu,\psi) \notag\\
  &\quad -  \frac{\bar a\, \Lip b}{8} \sum_{x,y,z} \abs{\nabla\psi(x,y)}^2 \bra[\big]{\Lambda(\mu)(x,z) + \Lambda(\mu)(y,z)} \notag\\
  &\geq \bra*{ d\,\underline{b}\,\underline{a} - \frac{d \, \bar a\, \Lip b}{2}} \cA(\mu,\psi) -  \frac{\bar a \, \Lip b}{2}I_2  \label{e:ZRP:II} \;.
\end{align}
where we used the Young inequality $uv\leq u^2/2 + v^2/2$.
\medskip

\emph{Third term $\tIII$:} For estimating the third term in~\eqref{e:curv:B}, denoted by $\tIII$, we use~\eqref{e:ZRP:DQ}, do a crude estimate to again apply~\eqref{e:logMean:homoIdent}
\begin{align*}
 \tIII &= \frac{1}{4} \sum_{x,y} \abs*{\nabla\psi(x,y)}^2 \Bigl( \partial_1\Lambda(\mu)(x,y) \mu_x  \bra*{b'(\mu_x) a(\mu_y) L_{\mu}\mu(x) + b(\mu_x) a'(\mu_y) L_{\mu}\mu(y)} \\
 &\qquad\qquad\qquad\qquad\qquad +  \partial_2\Lambda(\mu)(x,y) \mu_y  \bra*{b'(\mu_y) a(\mu_x) L_{\mu}\mu(y) +  b(\mu_y) a'(\mu_x) L_{\mu}\mu(x)} \Bigr) \\
 &\geq \inf_{\mu,x,y}\set*{ \frac{b'(\mu_x) L_{\mu}\mu(x)}{2b(\mu_x)} + \frac{a'(\mu_y) L_{\mu}\mu(y)}{2a(\mu_y)}} \cA(\mu,\psi) \ . 
\end{align*}
To bound the infimum, we observe that
\begin{equation*}
  \abs{L_{\mu} \mu(x)} \leq d\, \bar b \, \bar a \;.
\end{equation*}
Hence, in total we obtain
\begin{equation}
  \tIII \geq - \frac{d\, \bar b \, \bar a}{2}  \bra*{\frac{\Lip b}{\underline{b}} + \frac{\Lip a}{\underline{a}}} \cA(\mu,\psi) \;. \label{e:ZRP:III}
\end{equation}

\medskip

\emph{Fourth term $\tIV$:} For estimating the fourth term in~\eqref{e:curv:B}, denoted by $\tIV$, we use~\eqref{e:ZRP:M2} and compensate it partly by $\tI_2$ from~\eqref{e:ZRP:I2}
\begin{align}
 \tIV &= \frac{1}{2} \sum_{x,y,z} \nabla\psi(x,y) \nabla\psi(x,z) \Lambda(\mu)(x,z) \bra*{\mu_x b'(\mu_x) a(\mu_y) - \mu_y b(\mu_y) a'(\mu_x)} \notag\\
 &\geq - \frac{\bar a \, \Lip b + \bar b \Lip a}{2} \pra*{ d \, \cA(\mu, \psi) + \frac{1}{4} \sum_{x,y,z} \abs*{ \nabla\psi(x,y) }^2 \bra[\big]{ \Lambda(x,z) + \Lambda(y,z)} }  \;. \label{e:ZRP:IV}
\end{align}
\medskip 

\emph{Conclusion: } We combine all the estimates of the individual terms in~\eqref{e:curv:B} from the rewriting $\cB= \tI + \tII + \tIII + \tIV$. There is one small catch. After having applied the first bound~\eqref{e:ZRP:I} to $\tI$, we split for $\lambda \in (0,1)$ the non-negative part $\tI_1$ into $(1-\lambda)\tI_1$ and $\lambda \tI_1$, where only to the second term $\lambda \tI_1$ the bound~\eqref{e:ZRP:I2} is applied. The other three estimates \eqref{e:ZRP:II}, \eqref{e:ZRP:III} and \eqref{e:ZRP:IV} are applied in a straightforward manner to $\tII$, $\tIII$ and $\tIV$, respectively, to arrive at the lower bound 
\begin{align*}
  \cB(\mu,\psi)  &\geq (1-\lambda) \tI_1 + \bra*{\lambda\,\underline{a}\, \underline{b}- \frac{2\,\overline{a} \Lip b + \overline{b} \Lip a}{2} } \tI_2 \\
  &\quad + d\pra*{  \underline{a}\, \underline{b} - (1+\lambda)\frac{\bar{a}\, \bar{b}}{2} -  \frac{\bar a}{2} \bra*{ 2 +  \frac{ \bar b}{\underline{b}}} \Lip b -\frac{\overline{b}}{2} \bra*{1 +  \frac{\bar a}{\underline{a}}} \Lip a } \, \cA(\mu,\psi) \ .
\end{align*}
If $\lambda$ is chosen according
to~\eqref{e:ZRP:lambda} and by $\tI_1 \geq 0$, we arrive at the bound
$\cB(\mu,\psi)\geq \kappa \cA(\mu,\psi)$ with $\kappa$ given
in~\eqref{e:ZRP:kappa}. The final statement~\eqref{e:ZRP:kappa:asymptotic} follows by simple
calculus from the bound $\bar{a} \leq \underline{a} + \Lip a$,
similar for $\bar{b}\leq \underline{b} + \Lip b$ and observing that $\lambda = O(\eta)$ in this case. 
\end{proof}

\vspace{5mm}

\textbf{\underline{Acknowledgments}}: This work was supported by the
PHC Procope project "`Entropic Ricci curvature bounds of interacting
particle systems and their mean-field limits"'. MF was additionally
supported by ANR-11-LABX-0040-CIMI within the program
ANR-11-IDEX-0002-02, as well as Projects EFI (ANR-17-CE40-0030) and
MESA (ANR-18-CE40-006) of the French National Research Agency (ANR).
ME and AS were additionally supported by the German Research
Foundation (DFG) through the ``Hausdorff Center for Mathematics'' and
the CRC 1060 ``The Mathematics of Emergent Effects''.

\bibliographystyle{plain}
\bibliography{bib}

\end{document}